\documentclass[twoside,11pt,leqno]{article}

\usepackage{amsthm, amssymb, latexsym, amsmath, amscd, array, graphicx, enumerate, lmodern, slantsc, hyperref,amsfonts}
\usepackage[english]{babel}        
\usepackage[all]{xy}
\CompileMatrices

\usepackage{color}
\def\roji{}
\def\colorito{}
\def\rouge{}
\def\colorao{}
\def\azz{}
\def\roj{}
\def\mag{}
\def\kp{}
\def\az{}
\def\azul{}
\def\red{}
\def\rojo{}
\def\green{}
\def\magenta{}

\def\TC{\protect\operatorname{TC}}
\def\cat{\protect\operatorname{cat}}
\def\cone{\protect\operatorname{cl}}
\def\zcl{\protect\operatorname{zcl}}
\def\hdim{\protect\operatorname{hdim}}
\def\conn{\protect\operatorname{conn}}
\def\CW{\protect\operatorname{CW}}
\def\N{\protect\operatorname{N}}
\def\max{\protect\operatorname{max}}
\def\deg{\protect\operatorname{deg}}

\newtheorem{proposition}{Proposition}[section]
\newtheorem{definition}[proposition]{Definition}
\newtheorem{theo}[proposition]{Theorem}
\newtheorem{remark}[proposition]{Remark}
\newtheorem{ejem}[proposition]{Example}
\newtheorem{lema}[proposition]{Lemma}

\newtheorem{corollary}[proposition]{Corollary}

\makeindex

\title{The higher topological \rojo{complexity} of subcomplexes of products of spheres---and related polyhedral product spaces}
\author{Jes\'us Gonz\'alez, B\'arbara Guti\'errez, and Sergey Yuzvinsky}

\date{\today}

\begin{document}

\maketitle

\begin{abstract}
\rojo{We construct ``higher'' motion planners for automated systems whose space of states are homotopy equivalent to a polyhedral product space $Z(K,\{(S^{k_i},\star)\})$, \colorito{e.g.~robot arms with restrictions on the possible combinations of simultaneously moving nodes.} Our construction is shown to be optimal by explicit cohomology calculations. The higher topological complexity of other \colorao{families of} polyhedral product spaces is also determined.}
\end{abstract}



\tableofcontents

\section{Introduction}
\rojo{For a positive integer $s\in\mathbb{N}$, the $s$-th (higher or sequential) topological complexity of a path connected space $X$, $\TC_{s}(X)$, is defined in~\cite{Ru10} as the reduced Schwarz genus of the fibration $$e_s=e^X_{s}:X^{J_{s}}\to X^{s}$$ given by $e_{s}(f)=\left(\rojo{f_{1}}(1),\dots,\rojo{f_{s}(1})\right)$. Here $J_{s}$ denotes the wedge of $s$ copies of the closed interval $[0,1]$, in all of which $0\in[0,1]$ is the base point, and we think of an element $f$ in the function space $X^{J_{s}}$ as an $s$-tuple $f=(f_1,\ldots,f_s)$ of paths in $X$ all of which start at a common point. Thus, $\TC_{s}(X)+1$ is the smallest cardinality of open covers $\{U_i\}_i$ of $X^s$ so that, on each $U_i$, $e_s$ admits a section $\sigma_i$. In such a cover, $U_i$ is called a {\it local domain}, the corresponding section $\sigma_i$ is called a {\it local rule}, and the resulting family of pairs $\{(U_i,\sigma_i)\}$ is called a {\it motion planner}. The latter is said to be {\it optimal} if it has $\TC_{s}(X)+1$ local domains.}

\medskip\rojo{For practical purposes, the openness condition on local domains can be replaced (without altering the resulting numeric value of $\TC_{\colorito{s}}(X)$) by the requirement that local domains \colorito{are} pairwise disjoint Euclidean neighborhood retracts (ENR).}

\medskip
\rojo{Since} $e_s$ is the \rojo{standard} fibrational substitute of the diagonal \rojo{inclusion} $$d_s=d^{X}_{s}: X\hookrightarrow X^{s},$$ $\TC_{s}(X)$ coincides with the reduced Schwarz genus of \rojo{$d_{s}$. This suggests part (a) in the following definition, where we allow cohomology with local coefficients:}

\begin{definition} \rojo{Let $X$ be a connected space and $R$ be a commutative ring.}
\begin{itemize}\item[(a)] Given a positive integer $s$, we denote by $\zcl_s\, (H^*(X;\rojo{R}))$ the cup-length of elements in the kernel of the map induced by $d_s$ in cohomology. \rojo{Explicitly}, $\zcl_{s}\, (H^*(X;\rojo{R}))$ is the largest integer $m$ for which there exist cohomology classes $u_i \in H^*(X^s, A_i)$, where $X^s$ is the $s$-th Cartesian \rojo{power} of $X$ and \rojo{each} $A_i$ is a system of local coefficients, such that $\rojo{d_s^*(u_i)} =0$ for $i=1, \ldots, m\hspace{.3mm}$ and $\hspace{.5mm}0 \neq u_1 \otimes \cdots \otimes u_m \in H^*(X^s, \, A_1 \otimes \cdots \otimes A_m)$.
\item[(b)] The homotopy dimension of $X$, $\hdim (X)$, is the smallest dimension of $\CW$ complexes having the homotopy type of $X$. The connectivity of $X$, $\conn (X)$, \rojo{is the largest integer $c$ such that $X$ has trivial homotopy groups in dimensions at most $c$. We set $\conn(X)=\infty$ when no such $c$ exists.} 
\end{itemize}
\end{definition}

\begin{proposition}\label{ulbTCn}
For a path connected space $X$, $$ \zcl_{s}\,(H^*(X;\rojo{R})) \leq \TC_s (X) \leq \frac{s \, \hdim (X) }{\conn (X) +1}.$$
\kp{In particular for every path connected $X$,} $$\kp{\TC_s(X)\leq s\hdim(X).}$$
\end{proposition}

For a proof see~\cite[Theorem~3.9]{bgrt} \rojo{or, more generally,} \cite[Theorems 4 and 5]{Schwarz66}.  

\medskip
\rojo{The spaces we work with arise as follows. For} a positive integer \rojo{$k_i$} consider the \rojo{minimal} cellular structure \rojo{on} the \rojo{$k_i$}-dimensional sphere $S^{\rojo{k_i}}=e^0\cup e^{\rojo{k_i}}$. Here $e^0$ \red{is} the base point, which \rojo{is simply denoted} by $e$. Take the product \rojo{(\colorito{therefore} minimal)} \red{cell} decomposition \rojo{in}
$$
\rojo{\mathbb{S}{(k_1,\ldots, k_n)}:={S}^{k_1}\times\dots\times{S}^{k_n}}=\hspace{.3mm}\red{\bigsqcup_Je_J}
$$
whose cells $e_J$, indexed by subsets $J\subseteq[n]=\{1,\dots,n\}$, are defined as $e_J=\prod_{i=1}^ne^{d_i}$ where $d_i=0$ if $i \notin \azul{J}$ and $d_i=k_i$ if $i \in \azul{J}$. Explicitly, $$\red{e_J} =\left\{\red{(x_1,\ldots,x_n)}\in \rojo{\mathbb{S}(k_1,\ldots,k_n)}\;| \textrm{ $x_i=\red{e^0}$ if and only if $i\notin J$}\right\}.$$ 

\rojo{It is well known that the lower bound in Proposition~\ref{ulbTCn} is optimal for $\mathbb{S}(k_1,\ldots,k_n)$; Theorem~\ref{TC_s(X)} below asserts that the same phenomenon \rojo{holds} for subcomplexes. Note that, while $\mathbb{S}(k_1,\ldots,k_n)$ can be thought of as the configuration space of a mechanical robot arm whose $i$-th node moves freely in $k_i$ dimensions, a subcomplex $X$ of $\mathbb{S}(k_1,\ldots,k_n)$ encodes the information of the configuration space that results by imposing restrictions on the possible combinations of simultaneously moving nodes of the robot arm.}

\begin{theo}\label{TC_s(X)}
\rojo{A subcomplex $X\hspace{-.2mm}$ of $\hspace{.7mm}\mathbb{S}(k_1,\ldots,k_n)$ has} $\TC_s(X)=\zcl_{s}(H^*(X;\rojo{\mathbb{Q}}))$.
\end{theo}

Our methods imply that Theorem~\ref{TC_s(X)} could equally be stated using cohomology with coefficients in any ring of characteristic $0$. 

\medskip
\colorito{We provide an explicit description of $\zcl_{s}(H^*(X;\rojo{\mathbb{Q}}))$. The answer turns out to depend exclusively on the parity of the sphere dimensions $k_i$ (and on the combinatorics of the abstract simplicial complex underlying $X$). In order to better appreciate the phenomenon, it is convenient to focus first on the case where all the $k_i$ have the same parity\footnote{\colorito{An earlier version of the paper, signed by the current three authors, dealt only with the case when all the $k_i$ have the same parity. The unrestricted case was worked out later by the second named author using a mild variation of the original methods. Her results are included in the current updated version of the paper.}}. The corresponding descriptions,} in Theorems~\ref{tcsX odd} and~\ref{TC_s(X) even} as well as Corollary~\ref{sergey} \az{in the next section,} \colorito{generalize} those in~\cite{CohenPruid08,sergeypreprint}. \colorito{The unrestricted description is given in Subsection~\ref{stn4.1} (see~Theorem~\ref{TC_s(X), gene}). In either case,} the optimality of \colorito{the} cohomological lower bound will be a direct consequence of the fact that we actually construct an optimal motion planner. Our construction generalizes, in a high\kp{l}y non-trivial way, the one given \kp{first by \colorito{the third author} (\cite{Yu07}) for $s=2$ when $X$ is an arrangement complement, and then independently \az{by} Cohen-Pruidze (\cite{CohenPruid08}, as corrected in~\cite{morfismos}) in a more general case.}

\medskip
\rojo{By Hattori's work~\cite{hattori}, complements of generic complex hyperplane arrangements are up-to-homotopy examples of the spaces dealt with in Theorem~\ref{TC_s(X)} (with $k_i=1$ for all $i$). Those spaces are known to be formal, so their {\it rational} higher topological complexity has been shown in~\cite{carrasquel} to agree with the cohomological lower bound. Of course, such an observation can be recovered from Theorem~\ref{TC_s(X)} in view of the general fact that the rational topological complexity bounds from below the regular one. In any case it is to be noted that the rational higher TC agrees with the regular one for complements of generic complex hyperplane arrangements. Furthermore, these observations apply also for complements of the ``redundant'' arrangements considered in~\cite{CCX}, as well as for Eilenberg-Mac Lane spaces of \kp{all Artin type groups for finite groups generated by reflections, see \cite{sergeypreprint}}. In this direction, it is interesting to highlight that the agreement noted above between the rational higher TC and the usual one does not hold for other formal spaces. For instance, Lucile Vandembroucq has brought to the author's attention the fact that the rational $\TC_2$ of the symplectic group Sp$(2)$ is 2, one lower than its regular topological complexity.}

\medskip
\rojo{The bounds \rojo{in Proposition~\ref{ulbTCn}} for the higher topological complexity of a space easily yield Theorem~\ref{TC_s(X)} when all the $k_i$ agree with a fixed even number. If all the $k_i$ are even (but not necessarily equal), the result can still be proved with relative ease using the fact that the sectional category of a fibration is bounded from above by the cone-length of its base~(c.f.~\cite{LuMa}). \colorito{This} idea will be used in Section~\ref{othersec} in order to analyze the higher topological complexity of other polyhedral product spaces. But insisting on obtaining the required upper bound from the construction of explicit optimal motion planners (as we do) imposes a mayor task which, ironically, is much more elaborate when all the $k_i$'s are even. Yet, it seems to be extremely hard to give a proof of Theorem~\ref{TC_s(X)} that does not depend on the construction of an optimal motion planner if \colorito{at least one of the} $k_i$'s \colorito{is} odd.}

\section{\rojo{Optimal motion planners}}\label{secciondosmeramera}
\rojo{In this section we construct optimal motion planners for a subcomplex $X$ of $\mathbb{S}(k_1,\ldots,k_n)$} when all the \rojo{$k_i$'s} have the same parity. \rojo{We start by setting up some basic notation.}

\medskip
We \rojo{think of an element} $(b_1, b_2, \ldots, b_s) \in X^s$, with $b_j=(b_{1j}, \ldots, b_{nj}) \in X \subseteq \rojo{\mathbb{S}(k_1,\ldots,k_n)}$, as a matrix of size $n \times s$ whose \rojo{entry} $b_{ij}$ belongs to $\rojo{S^{k_i}}$ for all \rojo{$(i,j)\in [n]\times[s]$. (Here and below, for a positive integer $m$, $[m]$ stands for the initial integer interval $\{1,2,\ldots,m\}$, while $[m]_0$ stands for $[m]\cup\{0\}$).} Let $$\mathcal{P}=\{(P_1, \ldots, P_n)\,\, |\,\, P_i \,\, \mbox{is a partition of}\,\, \rojo{[s]} \,\, \mbox{for each}\,\, \rojo{i\in[n]}\,\}$$ be the set of $n$-\azz{tuples} of partitions of the interval \rojo{$[s]$. We assume that} elements $(P_1, \ldots, P_n) \in \mathcal{P}$ are ``ordered'' \rojo{in the sense that, if} $P_i=\{\alpha_1^i, \ldots, \alpha_{n(P_i)}^i\}$, \rojo{then} $L(\alpha_k^i) < L(\alpha_{k+1}^i)$ for $k \in \rojo{[n(P_i)-1]}$ where $L(\alpha_k^i)$ \kp{is defined as the smallest element of the set} $\alpha_{k}^i$. \rojo{In particular $1\in\alpha^i_1$.} \rojo{The norm of} each such $P=(P_1, \ldots, P_n) \in \mathcal{P}$ \rojo{is defined as}
\begin{equation}\label{lanormadeP}
|P|:=\sum_{i=1}^{n}(n(P_i)-1)=\sum_{i=1}^{n} |P_i|-n,
\end{equation}
the sum of all cardinalities of the partitions $P_i$ minus $n$. \rojo{We let}
$$
X^s_{P}=\Big\{(b_1, b_2, \ldots, b_s) \in X^s\,\, \Big| \!\!\!\!\!\begin{array}{ll} & \mbox{for each } i \in \rojo{[n]}, \;b_{ik}=\pm b_{i\ell} \mbox{ if and only \az{if}} \\ & \mbox{\az{both} $k$ \az{and} $\ell$ belong to the same part of $P_i$}
 \end{array} \Big\},
$$
\rojo{and say that an element $(b_1,b_2,\ldots,b_n)\in X^s_P$ has type $P$. Note that, if} $G:= \mathbb{Z}_2= \{1, -1\}$ \rojo{acts antipodally} on each sphere \rojo{$S^k$ and, for $x \in S^k$, $G\cdot x$ stands for the $G$-orbit of $x$, then}
\begin{equation}\label{Pcardinality}
|P_i|= \left|\{G\cdot b_{ij} \,\,|\,\,  \rojo{j\in[s]}\}\right| 
\end{equation}  
\rojo{for $(b_1, \ldots, b_s) \in X^s_{P}$ and $i \in[n]$. In addition, we consider $n$-tuples} $\beta=(\beta^1, \ldots, \beta^n)$ \rojo{of (possibly empty) subsets} $\beta^i \subseteq \alpha_1^i- \{1\}$ for $i \in\rojo{[n]}$, \rojo{and set} 
$$X^s_{P, \beta}= X^s_{P} \cap \left\{(b_1, \rojo{b_2,}\ldots, b_s) \in X^s \,\,|\,\, b_{i1}= b_{ik} \rojo{{}\Leftrightarrow{}} k \in \beta^i, \, \rojo{\forall\, (i,k)\in[n]\times \left([s]-\{1\}\right)}\right\}.$$  
\rojo{Note that the disjoint union decomposition}
\begin{equation}\label{X^s_P}
X^s_{P}=\bigsqcup_{\beta} X^s_{P, \beta},
\end{equation} 
\rojo{running over all $n$-tuples $\beta=(\beta^1, \ldots, \beta^n)$ as above, is \kp{topological,} \az{that is, the subspace topology in $X^s_P$ agrees with the so called \emph{disjoint union topology} determined by the subspaces $X^s_{P,\beta}$. In other words, a subset $U\subseteq X^s_P$ is open if and only if each of its pieces $U\cap X^s_{P,\beta}$ (for $\beta$ as above) is open in $X^s_{P,\beta}$. Needless to say, the relevance of this property comes from the fact that the continuity of a local rule on $X^s_P$ is equivalent to the continuity of the restriction of the local rule to each $X^s_{P,\beta}$.}} 

\subsection{Odd case}\label{casoimp}
\rojo{Throughout this subsection we assume that all $k_i$ are odd. We start by recalling an optimal motion planner for the sphere $\mathbb{S}(2d+1)=S^{2d+1}\hspace{.3mm}$---for which $\TC_s(\mathbb{S}(2d+1))=s-1$ \kp{a}s well known.}

\begin{ejem}\label{TCs S}{\em
\rojo{Local domains for $\mathbb{S}(2d+1)$ in the case $s=2$ are given by}
$$
A_0=\left\{(x, \rojo{-x}) \in \rojo{\mathbb{S}(2d+1)\times\mathbb{S}(2d+1)}\right\} \mbox{ \ \ and \ \ }A_1=\left\{(x, y) \in\rojo{\mathbb{S}(2d+1)\times\mathbb{S}(2d+1)} \,|\, x \neq -y\right\}
$$
\rojo{with corresponding} local rules $\phi_i$ $(i=0,1)$ described as follows: \rojo{For} $(x, -x)\in A_{0}$, $\phi_0 (x, \rojo{-x})$ \rojo{is} the \rojo{path} at constant speed from $x$ to $-x$ along the semicircle determined by $\nu(x)$, \rojo{where $\nu$ is some fixed  \kp{non-zero} tangent vector field of $\mathbb{S}(2d+1)$.} For $(x,y)\in A_1$, $\phi_1(x,y)$ \rojo{is}  the path at constant speed along the geodesic arc connecting $x$ with $y$. \rojo{To deal with the case $s>2$, we consider the domains $B_j$, $j \in[s-1]_0$, consisting of $s$-tuples} $(x_1, \ldots, x_s) \in\rojo{\mathbb{S}(2d+1)^s}$ \rojo{for which} $$\left\{\,k \in \{2, \ldots, s\}\,\,|\,\,x_1\neq -x_k\,\right\}$$ has cardinality $j$, \rojo{with local rules $\psi_j: B_j \rightarrow \mathbb{S}(2d+1)^{J_s}$ given by} $$\psi_j((x_1, \ldots, x_s))=\left(\psi_{j1}(x_1, x_1), \ldots, \psi_{js}(x_1, x_{s})\right)$$ \rojo{where} $\psi_{ji}(x_1, x_{i})= \phi_{\rojo{r}}(x_1, x_{i})$ \rojo{if} $(x_1, x_i)\in\rojo{A_r}$. \rojo{As shown in~\cite[Section~4]{Ru10}, the family $\{(B_j,\psi_j)\}$ is an optimal (higher) motion planner for $\mathbb{S}(2d+1)$.}
}\end{ejem}

A well known chess-board combination of the domains $B_j$ in Example~\ref{TCs S} yield domains for an optimal motion planner for the product $\mathbb{S}(k_1,\ldots,k_n)$ (see for instance the proof of Proposition~22 in page~84 of~\cite{Schwarz66}). \colorito{But the situation for an arbitrary subcomplex $X\az{{}\subseteq{}}\mathbb{S}(k_1,\ldots,k_n)$ is much more subtle.} 
\colorao{Actually, as} it will be clear from \az{the discussion} below, $\TC_s(X)$ is determined by \az{the} combinatorics of $X$ which we define \az{next}.

\medskip
First, for a given integer $s>1$,  \az{the $s$-norm of a finite (abstract) simplicial complex $\mathcal{K}$ is the integer invariant}
$$ 
\N^{\rojo{s}}(\mathcal{K}):=\displaystyle{\max \left\{ \, \rojo{\N_{\az{\mathcal{K}}}(J_1,J_2,\ldots,J_s)}
 \,\,|\,\, J_j \, \mbox{is a 
 simplex of } \mathcal{K}\,\, \text{for all}\,\, j \in \rojo{[s]} \right\}},
$$ 
\rojo{where}
\begin{equation}\label{nx} 
\displaystyle{\N_{\mathcal{K}}(J_1, J_2, \ldots, J_s): = \sum_{\ell=2}^s \left(\Big| \bigcap_{m=1}^{\ell-1} J_m- J_{\ell} \Big|+ \Big| J_{\ell} \Big| \right)}.
\end{equation}

\kp{Now we notice some properties of the above formulas  and give a simpler  more symmetric definition of $\displaystyle{\N_{\mathcal{K}}}$. \az{Start by observing that $\N_{\mathcal{K}}(J_1, J_2, \ldots, J_s)\leq\N_{\mathcal{K}}(J'_1, J'_2, \ldots, J'_s)$ provided $J_i\subseteq J'_i$ for $i\in[s]$. Consequently}
$$
\az{\N^{\rojo{s}}(\mathcal{K})=\displaystyle{\max \left\{ \, \rojo{\N_{\az{\mathcal{K}}}(J_1,J_2,\ldots,J_s)}
 \,\, |\,\, J_j \, \mbox{is a {\it maximal} 
 \hspace{.2mm}simplex of } \mathcal{K}\,\, \text{for all}\,\, j \in \rojo{[s]} \right\}},}
$$
\az{a formula that is well suited for the computation of $\N^s(\mathcal{K})$ in concrete cases.} Also let us put $I_{\ell}=\bigcap_{m=1}^{\ell-1} J_m- J_{\ell}$ for $\ell=2,3,\ldots,s$.} \az{Since $\bigcup_{\ell=2}^s I_{\ell}\subseteq J_1$ with $I_m\cap I_{m'}=\emptyset$ for every $m\not=m'$, we have:}

\begin{lema}\label{properties} 
\az{For (not necessarily maximal) simplexes $J_1,J_2,\ldots,J_s$ of $\mathcal{K}$,}
$$\kp{\displaystyle{\N_{\mathcal{K}}}(J_1, J_2, \ldots, J_s) = \sum_{\ell=2}^s\Big| I_{\ell}\Big|+ \sum_{\ell=2}^s\Big| J_{\ell}\Big|\leq \sum_{\ell=1}^s \Big| J_{\ell}\Big|.}$$
\end{lema}

\begin{proposition}\label{difdef}
\az{For $J_1,J_2,\ldots,J_s$ as above}
\begin{equation}\label{sy}
\kp{\displaystyle{\N_{\mathcal{K}}}(J_1, J_2, \ldots, J_s)= \sum_{\ell=1}^s| J_{\ell}|-\Big|\bigcap_{\ell=1}^s J_{\ell}
\Big|.} 
\end{equation}
\end{proposition}
\begin{proof}
\kp{Due to Lemma~\ref{properties} it suffices to \az{prove} the equality
$$\bigcup_{\ell=2}^sI_{\ell}=J_1-\bigcap_{\ell=1}^sJ_{\ell}.$$
\az{An element $x$ on} the left hand side (LHS) \az{satisfies} $x\in I_{\ell}$ for some $\ell\geq 2$ whence $x\not\in J_{\ell}$. Thus $x$ \az{lies on} the right hand side (RHS). Conversely, \az{for an element $x$ on} the RHS chose the smallest $\ell\geq 2$ such that $x\not\in J_{\ell}$.  By the choice of $\ell$
and definition of $I_{\ell}$ we have $x\in I_{\ell}$ whence $x$ \az{lies on} LHS.}
\end{proof}

\begin{corollary}
\kp{$\displaystyle{\N_{\mathcal{K}}}(J_1, J_2, \ldots, J_s)$ does not depend on the ordering of the set of simplexes.}
\end{corollary}

\kp{Now we apply the combinatorics we have developed to \az{a} CW subcomplex $X\az{{}\subseteq{}} \mathbb{S}(k_1,\ldots,k_n)$.}

\begin{definition}
\az{The index of $X$ is the (abstract)}
\kp{simplicial complex}
$$\kp{\mathcal{K}_X=\{J\subseteq[n] \;\;|\;\; e_J \mbox{ is a cell of $X$}\}.}$$
\az{For $d\in [n]$, we say that $X$ is $d$-pure (or simply pure, if $d$ is implicit) if its index is $d$-pure in the sense that all maximal simplexes of $\mathcal{K}_X$ have cardinality $d$.}
\end{definition}

\begin{remark}\label{gmac} \kp{\em Using the terminology from \cite{BahBende10}, $X$ is the \roj{polyhed\az{ral} product space} determined by the set of pairs $\az{\{}(S^{k_1},e),\ldots,(S^{k_n},e)\az{\}}$ and $\mathcal{K}_X$.}
\end{remark}

We use the notation $\displaystyle{\N_{X}}(J_1, J_2, \ldots, J_s)$ and $\N^{\rojo{s}}(X)$ for $\displaystyle{\N_{\mathcal{K}_{\az{X}}}}(J_1, J_2, \ldots, J_s)$ and $\displaystyle{\N^{\rojo{s}}(\mathcal{K}_{\az{X}}})$ respectively.

\smallskip\kp{Now we state one of \az{the} main results of the paper.}

\begin{theo}\label{tcsX odd}
\rojo{Assume all of the $k_i$ are odd. A subcomplex $X$ of the minimal CW cell structure on $\mathbb{S}(k_1,\cdots,k_n)$ has} $$\TC_s(X)=\N^{\rojo{s}}(X).$$
\end{theo}

\colorao{The proof of Theorem~\ref{tcsX odd} is deferred to the next sections; here we analyze its consequences and} interesting special instances, \colorao{starting with the case} when \az{$X$} \kp{is pure.} 

\begin{corollary}\label{purebis}
\kp{Suppose all of the $k_i$ are odd and $\az{X}$ is $\az{d}$-pure. Then}
\kp{$$\TC_s(X)=sd-\min\Big|\bigcap_{i=1}^sJ_i\Big|$$
where the minimum is taken over all sets $\{J_1,\ldots,J_s\}$ of maximal simplexes of ${\mathcal{K}_{\az{X}}}$.
In particular $\TC_s(X)\leq sd$ \az{with} equality if and only if $\hspace{.4mm}\bigcap_{i=1}^sJ_i$ \az{is empty}
for some choice of \az{maximal simplexes} $J_i$\az{'s}.}
\end{corollary}

\roj{Corollary~\ref{purebis} implies that, for $X$ \mag{$d$-pure,} \az{$\TC_s(X)$ growths linearly on $s$ provided $s$ is large enough. More \mag{precisely}, if} \kp{$w=w({\mathcal{K}_{\az{X}}})$ denotes the number of maximal simplexes \az{in $\mathcal{K}_{X}$, then} 
\begin{equation}\label{asintotico}
\TC_s(X)=d(s-w)+\TC_w(X)
\end{equation}
for $s\geq w$. More generally we have:}}

\begin{proposition}\label{crecilineagene}
\roj{Let $w$ be as above, \colorito{and set $d=1+\dim(\mathcal{K}_X)$.} \mag{Equation~$(\ref{asintotico})$ holds for any (pure or not) subcomplex $X$ of $\,\mathbb{S}(k_1,\ldots,k_n)$ as long as $s\geq w$.}}
\end{proposition}

\mag{The proof of Proposition~\ref{crecilineagene} uses the following auxiliary result:}

\begin{lema}\label{auxcrecilineagene}
\mag{In the setting of Proposition~$\ref{crecilineagene}$, if $J_1,\ldots,J_{w}$ are simplexes of $\mathcal{K}_X$ such that $\TC_{w}(X)=\sum_{i=1}^{w}|J_i|-\Big|\bigcap_{i=1}^{w}J_i\Big|$, then
$
\max\{\,|J_i|\;|\;i\in[w]\}=\colorito{d}.
$
}\end{lema}

\begin{proof}
\mag{\colorito{Assume} for a contradiction that $J_1,\ldots,J_{w}$ are simplexes of $\mathcal{K}_X$ such that $\TC_{w}(X)=\sum_{i=1}^{w}|J_i|-\left|\bigcap_{i=1}^{w}J_i\right|$ with $|J_i|<d$ for all $i\in[w]$. Choose a simplex $J_0$ of $\mathcal{K}_X$ with $|J_0|=d$, and indexes $i_1,i_2\in[w]$, $i_1<i_2$, with $J_{i_1}=J_{i_2}$. Set
$$
(J'_1,\ldots,J'_w):=(J_0,J_1,\ldots,J_{i_1-1},J_{i_1+1},\ldots,J_w).
$$
The contradiction comes from
$$
\N_X(J'_1,\ldots,J'_w)=\sum_{i=1}^w
|J'_i|-\Big|\bigcap_{i=1}^wJ'_i\Big|>\sum_{i=1}^w
|J_i|-\Big|\bigcap_{i=1}^wJ'_i\Big| 
\geq\sum_{i=1}^w
|J_i|-\Big|\bigcap_{i=1}^wJ_i\Big|=\TC_w(X)
$$
where the last inequality holds because $\bigcap_{i=1}^wJ'_i\subseteq\bigcap_{i=2}^wJ'_i=\bigcap_{i=1}^wJ_i$.}
\end{proof}


\begin{proof}[\roj{Proof of Proposition~$\ref{crecilineagene}$}] \roj{\mag{Let $s\geq w$.} Choose \colorao{maximal} simplexes $J'_1,\ldots,J'_s$ \mag{and $J_1,\ldots,J_w$} of $\mathcal{K}_X$ with $$\N^s(X)=\sum_{i=1}^s|J'_i|-\Big|\bigcap_{i=1}^sJ'_i\Big| \;\;\;\;\mag{\mbox{and}\;\;\;\,\N^w(X)=\sum_{i=1}^w|J_i|-\Big|\bigcap_{i=1}^wJ_i\Big|.}$$ Assume without loss of generality (since $s\geq w$) that $\{J'_1,\ldots,J'_s\}=\{J'_1,\ldots,J'_w\}$. Then
\begin{eqnarray*}
\TC_s(X)&=&\sum_{i=1}^{s}|J'_i|\,-\,\Big|\bigcap_{i=1}^{s}J'_i\Big|\;\;=\;\;\sum_{i=1}^{w}|J'_i|\,+\sum_{i=w+1}^{s}|J'_i|\,-\,\Big|\bigcap_{i=1}^{w}J'_i\Big|\\
&\leq&\TC_w(X)\,+\,\sum_{i=w+1}^{s}|J'_i|\;\;\leq\;\;\TC_w(X)\,+\,(s-w)d
\end{eqnarray*}
where, \colorito{as before,} $d=1+\dim(\mathcal{K}_{\colorito{X}})$. On the other hand, \mag{Lemma~\ref{auxcrecilineagene}} yields an integer $i_0\in[w]$ with $|J_{i_0}|=d$. Set $J_j:=J_{i_0}$ for $w+1\leq j\leq s$. Then
$$
\TC_w(X)+(s-w)d=\sum_{i=1}^{w}|J_i|\,-\,\Big|\bigcap_{i=1}^{w}J_i\Big|\,+\sum_{i=w+1}^{s}|J_i|=\sum_{i=1}^{s}|J_i|\,-\,\Big|\bigcap_{i=1}^{s}J_i\Big|
\leq\TC_s(X),
$$
completing the proof.}
\end{proof}

\mag{A more precise description of $\TC_s(X)$ can be obtained by imposing  condition\az{s} on $X$ which are stronger than purity.} \kp{\az{For instance, let $\mathbb{S}(k_1, \ldots, k_n)^{(d)}$ stand for the $d$-pure subcomplex of $\mathbb{S}(k_1, \ldots, k_n)$ with index} $\Delta[n-1]^{d-1}$, the $(d-1)$-skeleton of the full simplicial complex on $n$ vertices.} For instance, when $k_i=1$ for all $i\in[n]$, $\mathbb{S}(k_1, \ldots, k_n)^{(d)}$ is the $d$-dimensional skeleton in the minimal CW structure of the $n$-torus---the $n$-fold Cartesian product of $S^1$ with itself.

\begin{corollary}\label{sergey}
\rojo{If all of the $k_i$ are odd,} then $\TC_s\left(\mathbb{S}(k_1, \ldots, k_n)^{\rojo{(d)}}\right)=\min \{sd, (s-1)n\}$.
\end{corollary} 

In view of Hattori's theorem (\cite{hattori}, see also \cite[Theorem 5.21]{OT}), Corollary~\ref{sergey} specializes, with $k_i=1$ for all $i\in[n]$, to the assertion in~\cite[page~8]{sergeypreprint} describing the higher topological complexity of complements of complex hyperplane arrangements that are either linear generic, or affine in general position (cf.~\cite[Section~3]{Yu07}). It is also interesting to highlight that the ``min'' part in Corollary~\ref{sergey} (with $d=1$) can be thought of as a manifestation of the fact that, while the $s$-th topological complexity of an odd sphere is $s-1$, wedges of at least two spheres have $\TC_s=s$ ---just as any other nilpotent suspension space which is neither contractible nor homotopy equivalent to an odd sphere (\cite{minTC}). In addition, the ``min'' part in Corollary~\ref{sergey} detects a phenomenon not seen in terms of the Lusternik-Schnirelmann category since, as \colorao{indicated in Remark~\ref{downtocat} at the end of the paper, $\cat(\mathbb{S}(k_1, \ldots, k_n)^{(d)})=d$.}

\begin{proof}[Proof of Corollary~$\ref{sergey}$]
\az{Let $X$ stand for $\mathbb{S}(k_1,\ldots,k_n)^{\colorao{(d)}}$. For simplexes $J_1,\ldots,J_s$ of $\Delta[n-1]^{d-1}$, the} \kp{inequality $\N_X(J_1, \ldots, J_s)\leq \min\{sd,(s-1)n\}$ follows from 
Corollary~\ref{purebis} and Lemma \ref{properties} since $| I_{\ell}|+| J_{\ell}|\leq n$. \az{Thus $\TC_s\az{(X)})\leq\min\{sd,(s-1)n\}$} \roji{(notice \az{this holds for any $d$-pure} $X$).} \kp{To prove the opposite inequality suppose \az{first} that $sd\leq (s-1)n$, equivalently $n\leq s(n-d)$. Then
there exist a covering $\{C_1\ldots,C_s\}$ \az{of $[n]$} with $| C_k|=n-d$ for every $k\in[s]$. Put $J_k=[n]-C_k$ and 
notice that $J_k$ is a maximal simplex of $\az{\Delta[n-1]^{d-1}}$ for every $k$. \az{Further} $\bigcap_{k=1}^sJ_k=\emptyset$, \az{so that Corollary~\ref{purebis} yields} $$\TC_s\az{(X)}=sd=\az{\min\{sd,(s-1)n}\}.$$
Finally assume that $(s-1)n\leq sd$, i.e., $s(n-d)\leq n$. Then there exists a collection $\{C_1\ldots,C_s\}$ of mutually disjoint \az{subsets of $[n]$} with $| C_k|=n-d$ for every $k$. Put again $J_k=[n]-C_k$. We have
$$
\TC_s\az{(X)}\geq\sum_{\az{k=1}}^{\az{s}}|J_k|-\Big|\bigcap_{\az{k=1}}^{\az{s}}J_k\Big| =
sn-\sum_{\az{k=1}}^{\az{s}}|C_k|-\Big|\bigcap_{\az{k=1}}^{\az{s}}J_k\Big|=sn-\sum_{\az{k=1}}^{\az{s}}|C_k|-n+\Big|\bigcup_{\az{k=1}}^{\az{s}}C_k\Big|.
$$
\az{The result follows since the latter term simplifies to $(s-1)n=\min\{sd,(s-1)n\}$.}}}
\end{proof}

\magenta{\rojo{The higher topological complexity of a subcomplex $X$ of $\mathbb{S}(k_1,\ldots,k_n)$ whose index is \kp{pure but} not a skeleton depends heavily on the combinatorics of $\mathcal{K}_X$ ---and not just on its dimension. To illustrate the situation, we offer the following example.}}

\begin{ejem}{\em
\kp{Suppose the parameters are $n=4,\ d=2, \ s=3$;  $\mathcal{K}_1$ has the set of maximal simplexes $\{\{1,2\},\{2,3\},\{3,4\}\}$ while $\mathcal{K}_2$ the set $\{\{1,2\},\{1,\az{3}\},\{1,4\}\}$. \az{Fix positive odd integers $k_1,k_2,k_3,k_4$, and let} $X_i$ ($i=1,2$) be \az{the} CW \az{sub}complex of $\az{\mathbb{S}}(k_1,k_2,k_3,k_4)$ having $\mathcal{K}_i$ as its index. Then \az{Corollary~\ref{purebis}} gives $\TC_3(X_1)=6$
while $\TC_3(X_2)=5$.}}
\end{ejem}

\mag{Interesting phenomena can arise if $X$ is not pure.} \kp{This can be demonstrated by the following example\colorao{s}:}

\begin{ejem}{\em \colorao{
Take $s=n$. For $i\in [n]$, let $K_i=[n]-\{i\}$, and for $I \subseteq [n]$ , let $$W_I=\mathbb{S}(k_1, \ldots, k_n)^{(n-1)}-\bigcup_{i\in I} e_{K_i},$$ the subcomplex obtained from the fat wedge after removing the facets corresponding to vertices $i\in I$. As before, we assume that all of the $k_i$ are odd. Note that $W_I$ is $(n-1)$-pure if $|I|\leq1$, in which case Corollary~\ref{purebis} gives
\begin{equation}\label{purebutsamedim}
\TC_{\rojo{n}}(W_I)=n(n-1)-|I|.
\end{equation}
But the situation is slightly subtler when $2\leq|I|<n$ because, although the corresponding $W_I$ all have the same dimension, they fail to be pure, in fact:\begin{equation}\label{nonpurebutsamedim}
\TC_{n}(W_I)= \begin{cases}
n(n-1)-(\delta+1), & \text{if} \,\, |I|=2\delta+1 ; \\
n(n-1)-\delta, & \text{if} \,\, |I|=2\delta .
\end{cases}
\end{equation}
Note however that, by Corollary~\ref{sergey}, once all maximal simplexes have been removed from the fat wedge, we find the rather smaller value $\TC_n(W_{[n]})=n(n-2)$, back in accordance to~(\ref{purebutsamedim}). The straightforward counting argument verifying~(\ref{nonpurebutsamedim}) is left as an exercise for the interested reader; we just hint to the fact that the set of maximal simplexes of $\mathcal{K}_{W_I}$ is} $$\colorao{\left\{\,K_i \;|\; i \notin I\,\right\} \cup 
\left\{\,J \;|\; [n]-J \subseteq I \hspace{1.6mm}\mbox{and}\hspace{1.3mm} |J|= n-2 \,\right\}.}$$}
\end{ejem}

\begin{ejem}\label{allpossibilities}{\em 
\kp{Let $c_1>c_2$ be positive integers and $n=c_1+c_2$. Consider the simplicial complex $\mathcal{K}=\mathcal{K}^{c_1,c_2}$ \az{with vertices $[n]$} determined by two disjoint maximal simplexes $K_1$ and $K_2$ with $|K_1|=c_1$ and $|K_2|=c_2$. \az{Then, for any collection $J_1,\ldots,J_s$ of maximal simplexes of $\mathcal{K}$, where} precisely $s_1$ sets among $J_1,\ldots,J_s$ are equal to $K_1$ with \az{$0\leq s_1\leq s$, Proposition~\ref{difdef} yields}
$$
\N_{\az{\mathcal{K}}}(J_1,\ldots,J_s)=
\begin{cases}
(s-1)c_2, & s_1=0;\\ 
\kp{s_1c_1+\az{(s-s_1)}c_2}, & \kp{0<s_1<s;}\\
(s-1)c_1, & s_1=s.
\end{cases}
$$ 
This function of $s_1$ reaches its largest value when $s_1=s-1$ whence $\N^{\az{s}}({\mathcal{K}})=(s-1)c_1+c_2=sc_1-(c_1-c_2)$. \az{The latter} formula shows that, \colorito{as} $c_1-c_2$ runs through the integers $1,2,\ldots,c_1-1$, $\N^{\az{s}}({\mathcal{K}})$ runs through $sc_1-1,sc_1-2,\dots, (s-1)c_1+1$. \colorito{Whence,} due to Theorem~\ref{tcsX odd}, the same is true for  $\TC_s(X)$  where $X=X_{\az{c_1,c_2}}$ \az{is the} subcomplex of \az{some} $\az{\mathbb{S}}(k_1,\ldots,k_n)$ (with all $k_i$ odd) whose index equal\az{s} ${\mathcal{K}}$.}}
\end{ejem}

\begin{remark}{\em \colorito{The previous example} should be compared with the fact (proved in~\cite[Corollary~3.3]{bgrt}) that the $s$-th topological complexity of a given path connected space $X$ is bounded by $\cat(X^{s-1})$ from below, and by $\cat(X^s)$ from above.
\colorito{Example~\ref{allpossibilities} implies} that not only can both bounds be attained (with Hopf spaces in the former case, and with closed simply connected symplectic manifold in the latter) but any possibility in between \kp{can occur}. Indeed, \colorao{as indicated in Remark~\ref{downtocat} at the end of the paper,}
$\cat(X_{\az{c_1,c_2}}^p)=pc_1$ for every \az{positive integer $p$.}}
\end{remark}

\subsection{Proof of Theorem~\ref{tcsX odd}: the upper bound}\label{suboddmotion}
\rojo{The inequality $\N^{\rojo{s}}(X)\leq\TC_s(X)$ will be dealt with in Section~\ref{sectres} using cohomological methods; this subsection is devoted to establishing} the inequality $\TC_s(X)\leq \N^{\rojo{s}}(X)$ \rojo{by proving that the domains 
\az{\begin{equation}\label{unionxp}
D_j :=\bigcup X^s_{P},\quad j\in[\N^{\rojo{s}}(X)]_0,
\end{equation}
\rojo{where the union runs over those
$P \in \mathcal{P}$ with $|P|=j$ as 
defined in~(\ref{lanormadeP}),}}
give a cover of $X^s$ by pairwise disjoint ENR subspaces each of which admits a local rule---a section for $e_s$.}

\medskip
It is \rojo{easy} to see that \rojo{the $D_j$'s are pairwise disjoint. On the other hand, it follows from Proposition~\ref{interxp} below that~(\ref{unionxp}) is a topological disjoint union, so that~\cite[Proposition~IV.8.10]{Dold} and the obvious fact that each $X^s_P$ is an ENR imply the corresponding assertion for each $D_j$.}

\begin{lema}\label{coverodd}
$$ 
X^s = \bigcup_{j=0}^{\N^{\rojo{s}}(X)} D_j.
$$
\end{lema}
\begin{proof}
\rojo{Let $b\in X^s$,} say $b=(b_1, \ldots, b_s)\in e_{J_1} \times e_{J_2} \times \cdots \times e_{J_s} \az{{}\subseteq{}} X^s$, where $J_j \az{{}\subseteq{}} \rojo{[n]}$ for all \rojo{$j\in[s]$.} \colorito{Recall $G=\mathbb{Z}_2$ which acts antipodally on each sphere $S^{k_i}$.} \rojo{Note that} $$\sum_{i=1}^n \left|\left\{G\cdot b_{ij}\,\, |\,\, \rojo{j\in[2]}\right\}\right|-n= \left|\left\{\rojo{i\in[n]} \,\,|\,\, b_{i1} \neq \pm b_{i2}\right\}\right| \leq |J_1-J_2|+ |J_2|$$ \rojo{where the last inequality holds} since $\{\rojo{i\in[n]} \,\,|\,\, b_{i1} \neq \pm b_{i2} \} \subseteq J_1 \cup J_2$. \rojo{More generally,}
\begin{equation}\label{interpretacion1}
\sum_{i=1}^n |\{G\cdot b_{ij}\,\, |\,\, \rojo{j\in[s]}\}|-n= \sum_{\ell=2}^s |\{\rojo{i\in[n]} \,\,|\,\, b_{i\rojo{t}} \neq \pm b_{i\rojo{\ell}}\, \text{ for all } \, 1\leq t < \ell \}|
\end{equation}
where, for each $2 \leq \ell \leq s$,
\begin{equation}\label{interpretacion2}
|\{\rojo{i\in[n]}\,\, |\,\, b_{i\rojo{t}} \neq \pm b_{i\rojo{\ell}}\, \,\, \text{for all}\,\,  1\leq t < \ell \}| \leq | \bigcap_{\rojo{t}=1}^{\ell-1} J_{\rojo{t}} - J_{\ell}|+ |J_{\ell}|
\end{equation}
since \rojo{in fact} $$\{\rojo{i\in[n]}\,\, |\,\, b_{i\rojo{t}} \neq \pm b_{i\rojo{\ell}} \, \text{ for all }\, 1\leq t < \ell \} \subseteq \Big(\bigcap_{\rojo{t}=1}^{\ell-1} J_{\rojo{t}} \Big) \cup J_{\ell}.$$ Therefore,  \rojo{if $P=(P_1,\ldots,P_n)\in\mathcal{P}$ is the type of $b$, and we set $j=|P|$, then $b\in X^s_P\subseteq D_j$ where the inequality $j\leq\N^{s}(X)$ holds in view of~(\ref{Pcardinality}), (\ref{interpretacion1}), and~(\ref{interpretacion2}).}
\end{proof}


\rojo{Next,} in order to \rojo{construct a (well defined and continuous)} local section of $e_s$ over each $D_j$, \rojo{$j\in[\N^s(X)]$, we prove that~(\ref{unionxp}) is a topological disjoint union.}

\begin{proposition}\label{interxp}
\rojo{For any pair of elements} $P, P' \in \mathcal{P}$ with $|P|=|P'|$ and $P \neq P'$ \rojo{we have}
\begin{equation}\label{inter}
\overline{X^s_P} \cap X^s_{P'}= \emptyset = X^s_{P} \cap \overline{X^s_{P'}}.
\end{equation}
\end{proposition}
\begin{proof}
\rojo{Write} $P=(P_1, \ldots, P_n)$ and $P'=(P'_1, \ldots, P'_n)$ so that $$ \displaystyle{\sum_{i=1}^{n} |P_i|}=\displaystyle{\sum_{i=1}^{n} |P'_i|}.$$ If there exists \rojo{an integer $j_1\in[n]$ with} $|P_{j_1}| > |P'_{j_1}|$ \rojo{(or $|P_{j_1}| < |P'_{j_1}|$), then the hypothesis forces the existence of another integer $j_2\in[n]$ with} $|P_{j_2}| < |P'_{j_2}|$ \rojo{($|P_{j_2}| > |P'_{j_2}|$, respectively) and, in such a case~(\ref{inter}) obviously holds. Thus, without loss of generality we can assume} $|P_i|=|P'_i|$ for all \rojo{$i\in[n]$.} Since $P \neq P'$, there exists $k\in \rojo{[n]}$ such that $P_k \neq P'_k$. \rojo{Write} $P_k= \{\alpha_1, \ldots, \alpha_{\ell_0}\}$ and  $P'_k= \{\alpha'_1, \ldots, \alpha'_{\ell_0}\}$, both ordered in the sense \rojo{indicated at the beginning of the section}.

\medskip\rojo{Assume there are integers $t\in[\ell_0]$} with $L(\alpha_t) < L(\alpha'_t)$, \rojo{and let $t_0$ be the first such $t$ (necessarily $t_0>1$). Then any} $(b_1, \ldots, b_s) \in X^s_{P'}$ \rojo{must satisfy}
$$
b_{kL(\alpha_{\rojo{t_0}})} = \pm b_{kj_0}
$$ 
for some \rojo{$1 \leq j_0 \leq L(\alpha'_{t_0-1})\leq L(\alpha_{t_0-1})<L(\alpha_{t_0})$, condition that is then inherited by elements in $\overline{X^s_{P'}}$. However, by definition, any} $(b_1, \ldots, b_s) \in X^s_{P}$ \rojo{satisfies}
$$
b_{kL(\alpha_{\rojo{t_0}})}\neq \pm b_{kj}
$$
for all $1 \leq j <  L(\alpha_{\rojo{t_0}}).$ Therefore $X^s_P \cap \overline{X^s_{P'}}= \emptyset$. \rojo{A symmetric argument shows} $\overline{X^s_P} \cap X^s_{P'}= \emptyset$ \rojo{whenever there are integers $t\in[\ell_0]$ with $L(\alpha'_t) < L(\alpha_t)$. As a consequence, we can assume, without loss of generality, that} $L(\alpha_{j}) \leq L(\alpha'_{j})$ for all $j \in \rojo{[\ell_0]}\;$\rojo{{}---{this looses the symmetry, so we now have to make sure we show \emph{both} equations in~(\ref{inter}).}}

\medskip\noindent\rojo{{\bf Case $1$.} Assume there are integers $t\in[\ell_0]$} such that $L(\alpha_t) < L(\alpha'_t)$, and \rojo{let $t_0$ be the largest} \rojo{such $t$. We have already noticed that} $X^s_P \cap \overline{X^s_{P'}}= \emptyset$ \rojo{is forced}. Moreover, note that \rojo{either $t_0=\ell_0$ or, else,} $L(\alpha_{t_0}) < L(\alpha'_{t_0}) <L(\alpha'_{t_0+1})=L(\alpha_{t_0+1})$, \rojo{but in any case we have}
\begin{itemize}

\vspace{-.5mm}
\item if $(b_1, \ldots, b_s) \in X^s_{P}$, \rojo{then} $b_{kL(\alpha'_{t_0})}= \pm b_{kj_0}$ for some $1 \leq j_0 < L(\alpha'_{t_0})$, and

\vspace{-2mm}
\item if $(b_1, \ldots, b_s) \in X^s_{P'}$, then $b_{kL(\alpha'_{t_0})}\neq \pm b_{kj}$ for all $1 \leq j < L(\alpha'_{t_0})$.

\vspace{-2mm}
\end{itemize}

\noindent \rojo{Since the former condition is inherited on elements of $\overline {X^s_{P}}$, we see} $\overline{X^s_{P}}\cap X^s_{P'}= \emptyset$.

\medskip\noindent \rojo{{\bf Case $2$}. Assume}   $L(\alpha_j) = L(\alpha'_j)$ for all \rojo{$j\in[\ell_0]$. (Note \colorito{that} the symmetry is now restored.)} \rojo{Since $P_k\neq P'_k$, there is an integer $j_0\in[\ell_0]$ with} $\alpha_{j_0}\neq \alpha'_{j_0}$. Without \rojo{loss} of generality \rojo{we can further assume there is an integer $m_0 \in \alpha_{j_0}-\alpha'_{j_0}$ (note $m_0 \neq L(\alpha_{j_0})$, but once again the symmetry has been  destroyed). Under these conditions we have}
\begin{itemize}

\vspace{-.5mm}
\item if $(b_1, \ldots, b_s) \in X^s_P$, then $b_{kL(\alpha_{j_0})}= \pm b_{km_0}$, and

\vspace{-2mm}
\item if $(b_1, \ldots, b_s) \in X^s_{P'}$ then  $b_{kL(\alpha_{j_0})}=b_{kL(\alpha'_{j_0})}\neq \pm b_{km_0} $. 

\vspace{-2mm}
\end{itemize}

\noindent \rojo{Since the former condition is inherited on elements of $\overline {X^s_{P}}$, we see} $\overline{X^s_P} \cap X^s_{P'}= \emptyset$. Moreover, since $m_0 \notin \alpha'_{j_0}$, there \rojo{is $d_0\in[\ell_0]$} with $m_0 \in \alpha'_{d_0}$. \rojo{Necessarily $d_0\neq j_0$ and $m_0\notin\alpha_{d_0}$, so we now have} 
\begin{itemize}

\vspace{-.5mm}
\item if $(b_1, \ldots, b_s) \in X^s_{P'}$, then $b_{kL(\alpha'_{d_0})}=\pm b_{km_0}$, and

\vspace{-2mm}
\item if $(b_1, \ldots, b_s)\in X_P^s$, then $b_{kL(\alpha'_{d_0})}=b_{kL(\alpha_{d_0})} \neq \pm b_{km_0}$,

\vspace{-2mm}
\end{itemize}

\noindent \rojo{implying} $X^s_P\cap \overline{X^s_{P'}}=\emptyset$. 
\end{proof}

\rojo{Our only remaining task in this subsection is the construction of a local rule over $D_j$ for each $j\in[\N^s(X)]_0$. Actually, by~(\ref{X^s_P}), (\ref{unionxp}), and Proposition~\ref{interxp}, the task can be simplified to the construction of a local rule over each $X^{s}_{P,\beta}$. To fulfill such a goal, it will be convenient to normalize each sphere $S^{k_i}$} so to have great semicircles of length $1/2$. \rojo{Then, for} $x, y \in \rojo{S^{k_i}}$, we let $d(x, y)$ stand for the length of the shortest geodesic in \rojo{$S^{k_i}$} between $x$ and $y$ (e.g.~$d(x,-x)=1/2$). Likewise, the local rules $\phi_0$ and $\phi_1$ for \rojo{each $S^{k_i}$} defined at Example~\ref{TCs S} need to be adjusted---\rojo{but the domains $A_i$, $i=0,1$, remain unchanged---as follows:}
\roji{For} $i= 0,1$ \rojo{and}  $(x, y) \in A_i$ we set
$$\tau_i(x, y)(t)=\begin{cases}
\phi_i(x, y)\left(\frac{1}{d(x, y)}t\right), & 
0\leq \green{t < d(x, y)}; \\ y, & d(x, y)\leq t \leq 1.
\end{cases}$$ 
Thus, $\tau_i$ reparametrizes $\phi_i$ so to perform the motion at speed $1$, keeping still at the final position once it is reached---which happens at most at time $1/2$. 

\medskip
\rojo{In what follows it is helpful to keep in mind that, as before, elements} $(\rojo{b}_1, \ldots, \rojo{b}_s) \in X^s$, with $\rojo{b}_j=(\rojo{b}_{1j}, \ldots, \rojo{b}_{nj})$ for $j\in \rojo{[s]}$, \rojo{can be thought of as matrices $(b_{i,j})$ whose columns represent the various stages in $X$ through which motion is to be planned (necessarily along rows). Actually, we follow a ``pivotal'' strategy: starting at the first column, motion spreads to all other columns---keeping still in the direction of the first column. In detail, in terms of the notation set at the beginning of the introduction for elements in the function space $X^{J_s}$,} consider the map
\begin{equation}\label{elplanmot}
\rojo{\green{\varphi}\colon X^s} \to \mathbb{S}\rojo{(k_1,\ldots, k_n)}^{J_s}
\end{equation}
\rojo{given by} $\green{\varphi}\left((\rojo{b}_1, \ldots, \rojo{b}_s)\right)=\left(\green{\varphi_1}(\rojo{b}_1, \rojo{b}_1), \ldots , \green{\varphi_s}(\rojo{b}_1, \rojo{b}_s)\right)$ where, \rojo{for $j\in[s]$,} $$\green{\varphi_j}(\rojo{b}_1, \rojo{b}_j)= \left(\green{\varphi_{1j}} (\rojo{b}_{11}, \rojo{b}_{1j}), \ldots, \green{\varphi_{nj}}(\rojo{b}_{n1}, \rojo{b}_{nj})\right)$$ \rojo{is the path in $\mathbb{S}(k_1,\ldots,k_n)$, from $b_1$ to $b_j$, whose $i$-th coordinate} $\green{\varphi_{ij}}(\rojo{b}_{i1}, \rojo{b}_{ij})$, \rojo{\rojo{$i\in[n]$,} is the path in $S^{k_i}$, from $b_{i1}$ to $b_{ij}$,} defined by $$\green{\varphi_{i,j}} (\rojo{b}_{i1},\rojo{b}_{ij})(t)=\begin{cases}   \rojo{b}_{i1}, & 0 \leq t \leq t_{ \rojo{b}_{i1}}, \\ \rojo{\sigma}(\rojo{b}_{i1}, \rojo{b}_{ij})(t-t_{\rojo{b}_{i1}}), & t_{\rojo{b}_{i1}} \leq t \leq 1.  \end{cases}$$ Here $t_{\rojo{b}_{i1}}=\frac{1}{2} -d(\rojo{b}_{i1}, \rojo{e^0})$ and 
\begin{equation}\label{lasigma}
\rojo{\sigma(b_{i1}, b_{ij})}=\begin{cases}\tau_1\rojo{(b_{i1}, b_{ij})}, & (\rojo{b}_{i1}, \rojo{b}_{ij}) \in A_1; \\ \tau_0\rojo{(b_{i1}, b_{ij})}, & (\rojo{b}_{i1}, \rojo{b}_{ij}) \in A_0. \end{cases}
\end{equation}

\rojo{Fix $n$-tuples $P=(P_1, \ldots, P_{n}) \in \mathcal{P}$ and $\beta=(\beta^1, \ldots, \beta^n)$, with $P_i=\{\alpha_{1}^i, \ldots, \alpha_{n(P_i)}^i\}$ and $\beta^i \subseteq \alpha_{1}^i-\{1\}$ for all $i\in[n]$. Although \green{$\varphi$} is not continuous, its restriction $\green{\varphi}_{P, \beta}$ to $X^s_{P, \beta}$ is, for then~(\ref{lasigma}) takes the form}
$$
\sigma=\begin{cases}\tau_1, & \mbox{$j \notin \alpha^i_{1}$ or $j \in \rojo{\beta^i\cup\{1\}}$}; \\ \tau_0, & \mbox{$j \in \alpha^i_{1}$ and $j \notin \rojo{\beta^i\cup\{1\}}$}. \end{cases}
$$
\rojo{Since} $\green{\varphi}_{P, \beta}$ is \rojo{clearly} a section for the end-points evaluation map $e_s^{\rojo{\mathbb{S}(k_1,\ldots,k_n)}}$, we only \red{need to} check that $\green{\varphi}_{P, \beta}$ \rojo{actually} takes values in $X^{J_s}$, \rojo{i.e.~that our proposed motion planner does not leave $X$.}

\begin{remark}\label{intento}{\em
\rojo{An attempt to verify the analogous assertion in~\cite[proof of Proposition~3.5]{CohenPruid08} (where $s=2$), and the eventual realizing and fixing of the problems with that assertion, led to the work in~\cite{morfismos}. The verification in the current more general setting \colorito{(i.e.~proof of Proposition~\ref{noleave} below)} is inspired by the one carefully explained in~\cite[page~7]{morfismos}, and here we include full details for completeness.}
}\end{remark} 

\begin{proposition}\label{noleave}
\rojo{The image of $\green{\varphi}$ is contained in $X^{\green{J_s}}$.}
\end{proposition}
\begin{proof}
\magenta{\rojo{Choose $(b_1,b_2,\ldots,b_s)\in X^s$ where, as above, $b_j=(b_{1j},b_{2j},\ldots,b_{nj})\in X$. We need to check that, for all $j\in[s]$, the \rojo{image of} $\green{\varphi}_j(b_1,b_j)\colon[0,1]\to \mathbb{S}(k_1,\ldots,k_n)$ lies inside $X$. By construction, the path} $\green{\varphi}_{j}(b_1,b_j)$ \rojo{runs coordinate-wise,} from $b_1$ to $b_j$, according to the \rojo{instructions} $\tau_k(b_{i1}, b_{ij})$ ($k=0, 1, \,\, i \in [n] $), except that, \rojo{in the $i$-th coordinate,} the movement is delayed a time $t_{b_{i1}}\leq1/2$. The closer $b_{i1}$ gets to $\rojo{e^0}$, the closer the delaying time $t_{b_{i1}}$ gets to $1/2$. It is then convenient to think of the path $\green{\varphi}_j(b_1, b_j)$ as \rojo{running} in two \rojo{sections}. In the first \rojo{section} ($t\leq1/2$) all initial coordinates $b_{i1}=e^0$ keep still, while the rest of the coordinates (eventually) start traveling to their corresponding final position $b_{ij}$. Further, \rojo{when} the second section starts ($t=1/2$), any final coordinate $b_{ij}=\rojo{e^0}$ will already have been reached, \rojo{and will keep still throughout the rest of the motion.} As a result, \rojo{the image of} $\green{\varphi_j(b_1,b_j)}$ \rojo{is forced to be contained in} $X$. In more detail, let $e(J_1,\ldots,J_s):=e_{J_1} \times e_{J_2} \times \cdots \times e_{J_s} \az{{}\subseteq{}} X^{\green{s}}$ be \rojo{the} product of cells of $X$ \rojo{containing} $(b_1, b_2, \ldots, b_s )$. Then, coordinates  corresponding to indexes \rojo{$i \in [n]-J_1$} keep their initial position $b_{i1}=e^0$ through time $t\leq1/2$. Therefore $\green{\varphi}_j(b_1, b_j)[0,1/2]$ stays within $\rojo{\overline{e_{J_1}}}\subseteq X$. On the other hand, by construction, $\green{\varphi}_{ij}(b_{i1}, b_{ij})(t)=b_{ij}=e^0$ whenever $t\geq1/2$ and \rojo{$i \in[n]- J_j$.} Thus, $\green{\varphi}_j(b_1, b_j)[1/2,1]$ stays within $\rojo{\overline{e_{J_j}}}\subseteq X$.}
\end{proof}


 \subsection{Even case}\label{laspares}
We \rojo{now turn} our attention to the case when $X$ \rojo{is} a subcomplex of \rojo{$\mathbb{S}(k_1,\ldots,k_n)$ with all the $k_i$ even---assumption that will be in force throughout this subsection. As above,} the goal \rojo{is the construction of an optimal motion planner for the $s$-th topological complexity of $X$.} We \rojo{start with the following analogue of Example~\ref{TCs S}:} 

\begin{ejem}\label{TCs S even}{\em
\rojo{Local domains for the sphere $\mathbb{S}(2d)=S^{2d}$ in the case $s=2$ are given by}
\begin{eqnarray*}
\colorito{B_0}&=&\{(e^{\rojo{0}}, -e^{\rojo{0}}), (-e^{\rojo{0}}, e^{\rojo{0}})\} \rojo{{}\subseteq{}} \mathbb{S}\rojo{(2d)} \times \mathbb{S}\rojo{(2d)},\\
\colorito{B_1}&=&\{(x, -x) \in \mathbb{S}\rojo{(2d)} \times \mathbb{S}\rojo{(2d)} \,|\, x \neq \pm e^{\rojo{0}}\}, \mbox{ and}\\
\colorito{B_2}&=&\{(x, y) \in \mathbb{S}\rojo{(2d)} \times \mathbb{S}\rojo{(2d)} \,\,|\,\, x \neq -y\}\;\;=\;\;\mathbb{S}\rojo{(2d)} \times \mathbb{S}\rojo{(2d)} - (B_0 \cup B_1),
\end{eqnarray*}
\rojo{with corresponding local rules $\lambda_i\colon \colorito{B_i}\to \mathbb{S}(2d)^{[0,1]}$ ($i=0,1,2$) described as follows:} 
\begin{itemize}
\item $\lambda_0 (e^{\rojo{0}}, -e^{\rojo{0}})$ \rojo{and} $\lambda_0 (-e^{0}, e^{0})$ \rojo{are the} paths, at constant speed, from $e^{\rojo{0}}$ to $-e^{\rojo{0}}$ \rojo{and from $-e^{0}$ to $e^{0}$, respectively, along some fixed meridian---thinking of $e^0$ and $-e^0$ as the poles of $\mathbb{S}(2d)$.}

\vspace{-1.5mm}
\item \rojo{For} a \rojo{fixed nowhere zero} tangent vector field \rojo{$\upsilon$ on} $\mathbb{S}(2d)- \{\rojo{\pm}e^{\rojo{0}}\}$, $\lambda_1(x, -x)$ (with $x \neq\rojo{\pm}e^{\rojo{0}}$) \rojo{is} the path at constant speed from $x$ to $-x$ along the \rojo{great} semicircle determined by the tangent vector $\upsilon(x)$. 

\vspace{-1.5mm}
\item \rojo{For $x\neq-y$,} $\lambda_2(x, y)$ \rojo{is} the path \rojo{from $x$ to $y$,} at constant speed, along the \rojo{shortest} geodesic arc \rojo{determined by $x$ and $y$.}
\end{itemize}
}\end{ejem}

\rojo{The generalization of Example~\ref{TCs S even} to \colorito{the} higher topological complexity \colorito{of a} subcomplex of a product of even \colorito{dimensional} spheres is slightly more \colorito{elaborate} than the corresponding generalization of Example~\ref{TCs S} in the previous section due, in part,  to the additional local domain in Example~\ref{TCs S even}. So, before considering the general situation (Theorem~\ref{TC_s(X) even} below), and in order to  illustrate the essential points in our construction, it will be convenient to give full details in the case \colorito{of} $\TC_s(\mathbb{S}(2d))$.}

\smallskip
\rojo{Consider the sets}
\begin{eqnarray*}
T_0&=&\{(x_1, \ldots, x_s) \in \rojo{\mathbb{S}(2d)^s}\,\, | \,\, x_j\neq\rojo{\pm}e^{\rojo{0}},\,\, \text{for all}\,\, \rojo{j\in[s]}\},\\
T_1&=&\{(x_1, \ldots, x_s) \in \rojo{\mathbb{S}(2d)^s}\,\, | \,\, x_j=\rojo{\pm}e^{\rojo{0}},\,\, \text{for some}\,\, \rojo{j\in[s]}\}
\end{eqnarray*}
\rojo{and,} for each \rojo{partition $P$} of \rojo{$[s]$} and \rojo{each }$i \in \{0, 1\}$, $$\rojo{\mathbb{S}(2d)^{s}_{P, i}}=\left\{(x_1, \ldots, x_s) \in \rojo{\mathbb{S}(2d)^s}\,\, \Big| \!\!\!\!\begin{array}{ll} &x_{l}=\pm x_{k} \,\,  \mbox{if and only if} \,\,k\mbox{ and }l \\ & \mbox{belong to the same \rojo{part} in} \,\, \rojo{P} \,\,  \end{array} \right\} \cap T_i.$$ \rojo{The norm of the pair $(P,i)$ above is defined as} $\N(\rojo{P}, i)= |\rojo{P}|-i$. \rojo{Lastly,} for $k \in \rojo{[s]_0}$, \rojo{consider the set} 
\begin{equation}\label{lashaches}
H_k= \bigcup_{\N(\rojo{P}, i)=k}\rojo{\mathbb{S}(2d)^s_{P, i}}.
\end{equation}

\begin{proposition}\label{queva}
\rojo{There is an optimal motion planner for $\hspace{.172mm}\mathbb{S}(2d)$ with local domains $H_k$, $k\in[s]_0$.}
\end{proposition}
\begin{proof}
\rojo{The optimality of such a motion planner follows by the well known fact the $s$-th topological complexity of an even sphere is $s$. On the other hand, it is obvious that $H_0,\ldots, H_s$ form a pairwise disjoint covering of $\mathbb{S}(2d)^s$. Since each $\mathbb{S}(2d)^s_{P,i}$ is clearly an ENR, it suffices to show that~(\ref{lashaches}) is a topological disjoint union (so $H_k$ is also an ENR), and that each $\mathbb{S}(2d)^s_{P,i}$ admits a local rule (all of which, therefore, \colorito{determine} a local rule on $H_k$).}

\medskip\noindent 
\rojo{{\bf Topology of $H_k$}: For pairs} $(\rojo{P}, i)$ and $(\rojo{P}', i')$ \rojo{as above, with} $\N(\rojo{P}, i)=\N(\rojo{P'}, i')$ \rojo{and} $(\rojo{P}, i) \neq (\rojo{P'}, i')$, \rojo{we prove}
\begin{equation}\label{otravez}
\rojo{\overline{\mathbb{S}(2d)^s_{P,i}}\cap\mathbb{S}(2d)^s_{P',i'}=\emptyset=\mathbb{S}(2d)^s_{P,i}\cap\overline{\mathbb{S}(2d)^s_{P',i'}}.}
\end{equation}
If \rojo{$i\neq i'$, say} $i=1$ and $i'=0$, then \rojo{the first equality in~(\ref{otravez}) is obvious, whereas the second equality follows since} $|\rojo{P}| > |\rojo{P'}|$. \rojo{On the other hand, if} $i=i'$, \rojo{then} $|\rojo{P}|=|\rojo{P'}|$ with $\rojo{P}\neq\rojo{P'}$, and the \rojo{argument starting in the second paragraph of the proof of} Proposition \ref{interxp} gives~\rojo{(\ref{otravez})}. 

\medskip\noindent \rojo{{\bf Local section on $\mathbb{S}(2d)^s_{P,i}$}:} \rojo{We assume the partition $P=\{\alpha_1, \ldots, \alpha_n\}$ is ordered in the sense indicated at the beginning of this section.} For each $\rojo{\beta} \subseteq \rojo{\alpha_1}-\{1\}$, let  $$\rojo{\mathbb{S}(2d)^s_{P, i,\beta}}=\rojo{\mathbb{S}(2d)^s_{P, i}} \cap \{(x_1, \ldots, x_s) \in \rojo{\mathbb{S}(2d)^s}\, |\, x_1 =x_j\rojo{{}\Leftrightarrow{}} j \in \rojo{\beta},\;\rojo{\forall j\in[s]-1}\}. $$ \rojo{Since}
$$
\rojo{\mathbb{S}(2d)^s_{P, i}}=\bigsqcup_{\rojo{\beta} \subseteq \rojo{\alpha_1}-\{1\}}\rojo{\mathbb{S}(2d)^s_{P, i,\beta}}
$$
\rojo{is a topological disjoint union, it suffices to construct a local section on each $\mathbb{S}(2d)^s_{P, i,\beta}$.}

\medskip\noindent\rojo{{\bf Case $i=0$}.}
\rojo{As in the previous subsection, the required} local section \rojo{can be} defined \rojo{by the formula} $\sigma(x_1, \ldots, x_s)=(\sigma_1 (x_1, x_1), \ldots, \sigma_s(x_1, x_s))$ where
$$
\rojo{\sigma_j}=\begin{cases}
\rojo{\lambda_2}, & \mbox{if }j \in \rojo{\left([s]-\alpha_1\right) \cup \beta\cup\{1\}};\\
\rojo{\lambda_1}, & \mbox{\rojo{otherwise}.}
\end{cases}
$$

\medskip\noindent\rojo{{\bf Case $i=1$}.}
\rojo{The required local section is now defined in terms of the decomposition}
\begin{equation}\label{tdu2}
\rojo{\mathbb{S}(2d)^s_{P, i,\beta}}= \left( \rojo{\mathbb{S}(2d)^s_{P, i,\beta}} \cap \rojo{T}_0(\rojo{\alpha_1})\right)\sqcup \left(\rojo{\mathbb{S}(2d)^s_{P, i,\beta}} \cap \rojo{T}_1(\rojo{\alpha_1}) \right)
\end{equation} 
\rojo{which will be shown in Lemma~\ref{auxiliando} below to be a topological disjoint union. Here} $$\rojo{T}_0(\rojo{\alpha_1})=\{(x_1, \ldots, x_s) \in \rojo{\mathbb{S}(2d)^s}\,\, | \,\, x_j\neq\rojo{\pm}e^{\rojo{0}}, \,\, \mbox{for all}\,\, j\in \rojo{\alpha_1} \}$$ and $$\rojo{T}_1(\rojo{\alpha_1})=\{(x_1, \ldots, x_s) \in \rojo{\mathbb{S}(2d)^s}\,\, | \,\, x_j=\rojo{\pm}e^{\rojo{0}} ,\,\, \text{for some}\,\, j\in \rojo{\alpha_1} \}.$$ A local section \rojo{on} $\rojo{\mathbb{S}(2d)^s_{P, i,\beta} \cap T_0(\alpha_1)}$ is defined just as in \rojo{the case $i=0$}, \rojo{whereas} a local section \rojo{on} $\rojo{\mathbb{S}(2d)^s_{P, i,\beta} \cap T_1(\alpha_1)}$ is defined \rojo{by the formula} $\mu(x_1, \ldots, x_s)=(\mu_1 (x_1, x_1), \ldots, \mu_s(x_1, x_s))$ where
$$
\rojo{\mu_j}=\begin{cases}
\rojo{\lambda_2}, &\mbox{if \ }\rojo{j \in \left([s]-\alpha_1\right)\cup\beta\cup\{1\};}\\
\rojo{\lambda_0}, & \mbox{\rojo{otherwise.}}
\end{cases}
$$

\vspace{-8mm}
\end{proof}

\begin{lema}\label{auxiliando}
\rojo{The decomposition~$(\ref{tdu2})$ is a topological disjoint union $($recall $i=1).$}
\end{lema}
\begin{proof}
\rojo{The condition ``$x_j=\pm e^0$ for some $j\in\alpha_1$'' in $T_1(\alpha_1)$ is inherited by elements in its closure, in particular
$$
\left( \mathbb{S}(2d)^s_{P, i,\beta} \cap T_0(\alpha_1)\right)\sqcup \overline{\left(\mathbb{S}(2d)^s_{P, i,\beta} \cap T_1(\alpha_1)\right)}=\emptyset.
$$
On the other hand, since $i=1$, the condition ``$x_j=\pm e^0$ for some $j\not\in\alpha_1$'' is forced on elements of $\mathbb{S}(2d)^s_{P, i,\beta} \cap T_0(\alpha_1)$ and, consequently, on elements of its closure. But the latter condition is not fulfilled by any element in $\mathbb{S}(2d)^s_{P, i,\beta} \cap T_1(\alpha_1)$.}
\end{proof}

\rojo{We now focus on the general situation.}

\begin{theo}{\label{TC_s(X) even}}
\rojo{Assume all of the $k_i$ are even. A subcomplex $X$ of the minimal CW structure on $\mathbb{S}(k_1,\ldots,k_n)$ has} $$\TC_s(X)= \rojo{s(1+\dim(\mathcal{K}_X)).}$$
\end{theo}

\rojo{The inequality $s(1+\dim(\mathcal{K}_X))\leq\TC_s(X)$ will be dealt with in Section~\ref{sectres} using cohomological methods; in the rest of this subsection we prove  the inequality $\TC_s(X)\leq s(1+\dim(\mathcal{K}_X))$ by constructing an explicit motion planner with $1+s(1+\dim(\mathcal{K}_X))$ local domains---given by the sets in~(\ref{tdu3}) below.}

\medskip
\rojo{As in previous constructions, we think of an element} $(b_1, \ldots, b_s) \in X^s$ \rojo{with} $b_{j}=(b_{1j}, \ldots, b_{nj})$, $j \in [s]$, \rojo{as an $n\times s$ matrix whose $(i,j)$ coordinate is $b_{ij}\in\mathbb{S}(k_i)$}. \rojo{For} $P \in \mathcal{P}$ and $k\in\rojo{[n]_0}$, \rojo{set} $\N(P, k):= \sum_{i=1}^n |P_i|-k$, the norm of the pair $(P,k)$, and 
$$
X^{s}_{P, k}:=X^s_P \cap \left\{(b_{1}, \ldots, b_s) \in  \rojo{\mathbb{S}(k_1,\ldots,k_n)^s} \;\; \Big| \!\!\!\!\!\begin{array}{cl}&\,(b_{i1}, \ldots, b_{is}) \rojo{{}\in T_{1,k_i}} \mbox{ for}\\ &\text{ exactly $k$ indexes $i \in \rojo{[n]}$}\end{array} \right\}
$$
\rojo{where $T_{1,k_i}=\left\{(x_1, \ldots, x_s) \in \rojo{\mathbb{S}(k_i)^s}\,\, | \,\, x_j=\rojo{\pm}e^{\rojo{0}},\,\, \text{for some}\,\, \rojo{j\in[s]}\right\}$. The local domains we propose are given by}
\begin{equation}\label{tdu3}
W_r= \bigcup_{\N(P, k)=r} X^{s}_{P, k}.
\end{equation}
\rojo{By~(\ref{Pcardinality}), the norm $\N(P, k)$} \rojo{is the number of} ``row'' \rojo{$G$-}orbits \rojo{different from that of $e^0$ in any matrix $(b_1,\ldots,b_s)\in X^s_{P,k}$}. Therefore \rojo{the sets $W_r$ with $r\in[s(1+\dim(\mathcal{K}_X))]_0$ yield a pairwise disjoint cover of} $X^s$. \rojo{Our task then is to show:}

\begin{proposition}\label{tarea1.1}
\rojo{Each $W_r$ is an ENR admitting a local rule.}
\end{proposition}

Our proof of Proposition~\ref{tarea1.1} \colorito{depends on showing that~(\ref{tdu3}) is a topological disjoint union (Lemma~\ref{separadotes} below) and that each piece $X^s_{P,k}$ admits a suitably finer topological decomposition ((\ref{tdu4}),~(\ref{tdu10}), and Proposition~\ref{construfinal} below).}

\begin{lema}\label{separadotes}
\rojo{For} $P, P' \in \mathcal{P}$ and $k , k' \in \rojo{[n]_0}$ \rojo{with} $N(P, k)= N(P', k')$ and $(P, k) \neq (P', k')$, $$\overline{X^s_{\rojo{P,k}}}\cap X^s_{\rojo{P', k'}}= \emptyset = X^s_{\rojo{P, k}}\cap \overline{ X^s_{\rojo{P',k'}}} $$
\end{lema}  
\begin{proof}
\rojo{Write} $P=(P_1, \ldots, P_{\rojo{n}})$ and $\rojo{P'}=(P_1', \ldots, P'_{\rojo{n}})$ \rojo{so that, by hypothesis,} $\sum_{i=1}^n |P_i| -k = \sum_{i=1}^n |P'_i| -k'.$ If $k > k'$, then  $\overline{X^s_{\rojo{P,k}}}\cap X^s_{\rojo{P',k'}}= \emptyset$, \rojo{and since} $\sum_{i=1}^n |P_i| > \sum_{i=1}^n |P'_i|$ \rojo{is forced, we also get} $X^s_{\rojo{P, k}}\cap \overline{ X^s_{\rojo{P',k'}}} = \emptyset$. If $k=k'$, then $|P|= |P'|$ with $P \neq P'$ and, \rojo{just as for~(\ref{otravez}), the argument starting in the second paragraph of the proof of Proposition \ref{interxp} yields the conclusion.}
\end{proof}

\rojo{Next we work with a fixed pair} $(P, k) \in \mathcal{P} \times \rojo{[n]_0}$ with $P=(P_1, \ldots, P_n)$ \rojo{and where each} $P_i=\{\alpha_1^i, \ldots, \alpha^{i}_{n(P_i)}\}$ \rojo{is ordered as described at the beginning of this section.} For \rojo{a subset} $I\subseteq [n]$ \rojo{consider} the set $T_I=\{ (b_1, \ldots, b_s) \in X^s\, \, | \,\, (b_{i1}, \ldots, b_{is})\rojo{{}\in T_{1,k_i}} \, \,\text{if and only if}\,\, i \in I\}.$ \rojo{Then}~(\ref{X^s_P}) \rojo{yields a topological disjoint union}
\begin{equation}\label{tdu4}
X^s_{P, \, k}= \bigsqcup_{\beta, I} \left(X^s_{P, \beta} \cap T_I\right)
\end{equation}
\rojo{running over subsets $I\subseteq[n]$ of cardinality $k$, and $n$-\azz{tuples} $\beta=\rojo(\beta^1, \ldots, \beta^n\rojo)$ of (possibly empty) subsets} $\beta^i \subseteq \alpha^i_1- \{1\}$. \rojo{Besides, as suggested by~(\ref{tdu2}) in the proof of Proposition~\ref{queva}, it is convenient to decompose even further each piece in~(\ref{tdu4})}. For each $i \in \rojo{[n]}$, let 
\begin{eqnarray}\label{T_01}
\rojo{T}_{0}( \alpha^i_1)&=&\{(b_1, \ldots, b_s) \in \rojo{X^s}\, \, | \,\, b_{ij} \neq \rojo{\pm} e^{\rojo{0}} \mbox{ for all }j \in \alpha^i_1  \}, \nonumber\\ 
\rojo{T}_{1}( \alpha^i_1)&=&\{(b_1, \ldots, b_s) \in \rojo{X^s}\, \, | \,\, b_{ij} = \rojo{\pm}e^{\rojo{0}}\,\, \text{for some}\,\,j \in \alpha^i_1 \}
\end{eqnarray}
\rojo{and, for $I\colorito{{}= \{\ell_1, \ldots, \ell_{|I|}\}}\subseteq[n]$ and $\varepsilon =(t_1, \ldots, t_{|I|})\in\{0,1\}^{|I|}$,}
$$T_{\varepsilon}(I)= \rojo{T_I\cap{}}\bigcap^{\colorito{|I|}}_{\colorito{ i =1} } \rojo{T}_{t_i}(\alpha^{\colorito{\ell_i}}_1).$$ \rojo{In these terms there is an additional topological disjoint union decomposition}
\begin{equation}\label{tdu10}
X^s_{P, \beta} \cap T_I = \bigsqcup_{\varepsilon \in\rojo{\{0,1\}^{|I|}}} \left(X^s_{P, \beta} \cap T_{\varepsilon}(I)\right).
\end{equation}

\rojo{Proposition~\ref{tarea1.1} is now a consequence of~(\ref{tdu4}),~(\ref{tdu10}), Lemma~\ref{separadotes}, and the following result:} 

\begin{proposition}\label{construfinal}
\rojo{For $P$, $\beta$, I, and $\varepsilon$ as above, $X^s_{P, \beta} \cap T_{\varepsilon}(I)$ is an ENR admitting a local rule.}
\end{proposition}
\begin{proof}
\rojo{The ENR property follow since, in fact, $X^s_{P, \beta} \cap T_{\varepsilon}(I)$ is homeomorphic to the Cartesian product of a finite discrete space and a product of punctured spheres. Indeed, the information encoded by $P$ and $\beta$ produces the discrete factor, as coordinates in a single $G$-orbit are either repeated (e.g.~in the case of $\beta$) or sign duplicated. Besides, after ignoring such superfluous information as well as all $e^0$-coordinates (determined by $I$ and $\varepsilon$), we are left with a product of punctured spheres.}

\medskip
\rojo{The needed local rule can be defined following the algorithm at the end of Subsection~\colorito{\ref{suboddmotion}.} Explicitely, let $\rho_i$ ($i=0,1,2$) denote the local rules obtained by normalizing the corresponding $\lambda_i$ (defined in Example~\ref{TCs S even}) in the same manner as the local rules $\tau_i$ were obtained right after the proof of Proposition~\ref{interxp} from the corresponding $\phi_i$. Then consider the (non-continuous) global section $\varphi\colon X^s \to \mathbb{S}\rojo{(k_1,\ldots, k_n)}^{J_s}$ defined through the algorithm following~(\ref{elplanmot}), except that~(\ref{lasigma}) gets replaced by $$\sigma(b_{i1}, b_{ij})=\rho_m(b_{i1}, b_{ij}), \text{ if } (b_{i1}, b_{ij}) \in \colorito{B_m}\text{ for }m\in\{0,1,2\}$$ where the domains $\colorito{B_m}$ are now those defined in Example~\ref{TCs S even}. As in the previous subsection, the point is that the restriction of $\varphi$ to $X^s_{P, \beta} \cap T_{\varepsilon}(I)$ is continuous since, in that domain, the latter equality can be written as
$$
\sigma=\begin{cases}
\rho_2, &\mbox{if \ }j \in \left([s]-\alpha_1^i\right)\cup\beta^i\cup\{1\};\\
\rho_1, &\mbox{if \ }j \in \alpha_1^i - \left(\beta^i\cup\{1\}\right)\text{ and } t_i=0;\\
\rho_0, & \mbox{if \ }j \in \alpha_1^i - \left(\beta^i\cup\{1\}\right)\text{ and } t_i=1.
\end{cases}
$$
In addition, the proof of Proposition~\ref{noleave} applies word for word to show that the image of $\varphi$ is contained in $X^{J_s}$.}
\end{proof}

\begin{remark}{\em
\colorao{The gap noted in Remark~\ref{intento} also holds in~\cite{CohenPruid08} when all the $k_i$ are even. The new situation is subtler in view of an additional gap (pinpointed in~\cite[Remark~2.3]{morfismos}) in the proof of~\cite[Theorem~6.3]{CohenPruid08}. Of course, the detailed constructions in this section fix the problem and generalize the result.}}
\end{remark}

\section{Zero-divisors cup-length}\label{sectres}
\rojo{We now show that, for a subcomplex $X$ of $\mathbb{S}(k_1,\ldots,k_n)$ where all the $k_i$ have the same parity, the cohomological lower bound for $\TC_s(X)$ in Proposition~\ref{ulbTCn} is optimal and agrees with the upper bound coming from our explicit motion planners in the previous section. Throughout this section we use cohomology with rational coefficients, writing $H^*(X)$ as a shorthand of $H^*(X;\mathbb{Q})$.} 

\medskip
\rojo{Recall $H^*(\mathbb{S}(k_1,\ldots,k_n))$} is an exterior algebra \rojo{$E(\epsilon_1,\ldots,\epsilon_n)$ where $\epsilon_i$ corresponds to the $\mathbb{S}(k_i)$ factor, so that $\deg(\epsilon_i)=k_i$.} For $J=\{j_1, \ldots, j_k\} \rojo{{}\subseteq [n]}$, let $\epsilon_J= \epsilon_{j_1}\cdots\epsilon_{j_k}$. \rojo{The cohomology ring $H^*(X)$ is a quotient of $E(\epsilon_1,\ldots,\epsilon_n)$:}

\begin{proposition}{\label{H^*(X)}}
\rojo{For} a subcomplex $X$ of the \rojo{minimal} CW-decomposition of \rojo{$\hspace{.4mm}\mathbb{S}(k_1\ldots,k_n)$,} the cohomology ring $H^*(X)$ is the quotient of the exterior algebra \rojo{$E(\epsilon_1,\ldots,\epsilon_n)$} by the \kp{monomial} ideal $I_X$ generated by those $\epsilon_J$ \rojo{for which} $e_J$ is not a cell of $X$. 
\end{proposition}

For a proof \rojo{(in a more general context)} of this proposition see \cite[Theorem 2.35]{BahBende10}. \rojo{In particular, an additive basis for $H^*(X)$ is given by the products $\epsilon_J$ with $e_J$ a cell of $X$. We will work with the corresponding tensor power basis for $H^*(X^s)$.}

\begin{remark}\label{elhandlingparodd}\rouge{\em
\colorito{In} the next two results, the hypothesis of having a fixed parity for all the $k_i$ \colorito{will be} crucial \colorito{when handling} products of zero divisors in $H^*(X^s)$. Indeed, a typical such element has the form
$$
z=c_1\cdot\epsilon_i\otimes1\otimes\cdots\otimes1+
c_2\cdot1\otimes\epsilon_i\otimes1\otimes\cdots\otimes1+\cdots+
c_s\cdot1\otimes\cdots\otimes1\otimes\epsilon_i
$$
for $i\in[n]$ and $c_1,\ldots,c_s\in\mathbb{Q}$ with $c_1+\cdots+c_s=0$. Then, by graded commutativity, $z^2$ is forced to vanish when $k_i$ is odd. However $z^s\neq0$ if $k_i$ is even and $c_j\neq0$ for all $j\in[s]$. 
}\end{remark}

\rojo{Proposition~\ref{ulbTCn} and the following result complete the proof of Theorem~\ref{tcsX odd}.}

\begin{proposition}\label{cotaabajoimpar}
Let $X$ \rojo{be as in Proposition~$\ref{H^*(X)}$. If all of the $k_i$ are odd,} then $$\N^{\rojo{s}}(X)\leq\zcl_s(H^*(X)).$$
\end{proposition}
\begin{proof}
Let $H_X=H^*(X^s)=[H^*(X)]^{\otimes s}$. For $u \in H^*(X)$ and $2 \leq \rojo{\ell} \leq s$, let $$u(\rojo{\ell})=\underbrace{u \otimes 1 \otimes \cdots \otimes 1}_{s \text{ \rojo{factors}}} -\underbrace{1 \otimes \cdots \otimes 1 \otimes \overset{\,\rojo{\ell}}{u} \otimes 1 \otimes \cdots \otimes 1}_{s \text{ \rojo{factors}}} \in H_{X} $$ where an $\rojo{\ell}$ on top of a tensor factor indicates the coordinate where the factor appears. \rojo{Take a cell} $e_{J_1} \times e_{\rojo{J}_2} \times \cdots \times e_{J_s} \az{{}\subseteq{}}  X^s $, $J_1, \ldots, J_s\rojo{{}\subseteq[n]}$. For $2 \leq \ell \leq s$, let
\begin{center}
$\begin{array}{ll}
\gamma( J_1, \ldots, J_{\ell}) &= \displaystyle{\prod_{j \in \Big(\bigcap_{m=1}^{\ell-1} J_m - J_{\ell} \Big) \cup J_{\ell} } \epsilon_j (\ell)}
\\
& =\displaystyle{ \sum_{\phi_{\ell} \subseteq \Big(\bigcap_{m=1}^{\ell-1} J_m - J_{\ell} \Big) \cup J_{\ell}} \pm \epsilon_{\phi^c_{\ell}} \otimes 1 \otimes \cdots \otimes 1 \otimes \overset{\ell\hspace{2mm}}{\epsilon_{\phi_{\ell}}} \otimes 1 \otimes \cdots \otimes 1}
\end{array}$
\end{center}
where $\phi_{\ell}^c$ \rojo{stands for} the complement of $\phi_{\ell}$ in $\Big(\bigcap_{m=1}^{\ell-1} J_m - J_{\ell} \Big) \cup J_{\ell}$. \rojo{It suffices to prove the non-triviality of the} product of $\N_X(J_1, \ldots, J_s)$ zero-divisors 
\begin{equation}\label{prodzcls}
\gamma (J_1, J_2) \cdots \gamma (J_1, \ldots, J_s)= \sum_{\phi_ 2, \ldots, \phi_s} \pm \epsilon_{\phi_2^c}\cdots\epsilon_{\phi^c_s} \otimes \epsilon_{\phi_2} \otimes \cdots \otimes \epsilon_{\phi_s}
\end{equation} 
where the sum runs over all $\phi_\ell \rojo{{}\subseteq{}} \Big(\bigcap_{m=1}^{\ell-1} J_m - J_{\ell} \Big) \cup J_{\ell} $ with $2 \leq \ell \leq s$. \rojo{With this in mind, note that} the term
\begin{equation}\label{eltermino}
\pm \epsilon_{J_1- J_2}  \cdots \epsilon_{(J_1 \cap \cdots \cap J_{\ell-1})-J_{\ell}} \cdots \epsilon_{(J_1 \cap \cdots \cap J_{s-1})-J_s} \otimes \epsilon_{J_2} \otimes \cdots \rojo{{}\otimes \epsilon_{J_\ell}\otimes\cdots}\otimes\epsilon_{J_s},
\end{equation}
\rojo{which} appears in~(\ref{prodzcls}) \rojo{with} $\phi_\ell=J_\ell$ for $2 \leq  \ell \leq s$, is \rojo{a basis element} because $$\epsilon_{J_1- J_2}  \cdots \epsilon_{(J_1 \cap \cdots \cap J_{\ell-1})-J_\ell} \cdots \epsilon_{(J_1 \cap \cdots \cap J_{s-1})-J_s}= \epsilon_{J_0}$$ with $J_0 \az{{}\subseteq{}} J_1$. \rojo{The non-triviality of~(\ref{prodzcls}) then follows by observing that~(\ref{eltermino})} cannot arise when other summands \rojo{in~(\ref{prodzcls})} are expressed in terms of \rojo{the} basis for $H_X$. \rojo{In fact, each summand
\begin{equation}\label{otrossumandos}
\pm \epsilon_{\phi_2^c} \cdots \epsilon_{\phi^c_s} \otimes \epsilon_{\phi_2} \otimes \cdots \otimes \epsilon_{\phi_s}
\end{equation}
in~(\ref{prodzcls}) is either zero or a basis element and, in the latter case,~(\ref{otrossumandos}) agrees (up to sign) with~(\ref{eltermino}) only if $\phi_{\ell}=J_{\ell}$ for $\ell=2,\ldots,s$.}
\end{proof}

\rojo{Likewise, the proof of Theorem~\ref{TC_s(X) even} is complete by Proposition~\ref{ulbTCn} and the following result:}

\begin{proposition}\label{casiultima}
\rojo{Let $X$ be as in Proposition~$\ref{H^*(X)}$. If all of the $k_i$ are even, then} $$s \rojo{\left(1+\dim(\mathcal{K}_X)\right)} \leq \zcl_s (\rojo{H^*(X)}).$$
\end{proposition}  
\begin{proof}
For $u \in H^*(X)$, set $$\rojo{\overline{u}}= \left(\hspace{.5mm}\sum_{i=1}^{s-1} 1 \otimes \cdots \otimes 1 \otimes \overset{i}{u} \otimes 1 \otimes \cdots \otimes 1\right)- 1 \otimes \cdots \otimes 1 \otimes (s-1) u\rojo{{\,}\in H_X.}$$
\rojo{Fix} a maximal cell $e_L$ of $X$ where $L=\{\delta_1, \ldots, \delta_{\ell}\} \subseteq \rojo{[n]}$ \rojo{(so $\ell=1+\dim(\mathcal{K}_X)$)}. \rojo{A straightforward calculation yields,} for \rojo{$i\in[\ell]$,}
$$
(\overline{\epsilon_{\delta_i}})^s=(1-s)s!(\underbrace{\epsilon_{\delta_i} \otimes \cdots \otimes \epsilon_{\delta_i}}_{s \text{ \rojo{factors}}}),
$$
so
$$
\prod_{i=1}^{\ell} (\overline{\epsilon_{\delta_i}})^s= \rojo{\left((1-s)s!\right)^\ell} \underbrace{\epsilon_L \otimes \cdots \otimes \epsilon_L}_{s \text{ \rojo{factors}}}
$$
\rojo{which is a nonzero product of $s\ell$ zero-divisors in $H_X$.}
\end{proof}

\begin{remark}{\em
\rojo{The estimate $s(1+\dim(\mathcal{K}_{\colorao{X}}))\leq\TC_s(X)$ can also be obtained by noticing that, in the notation of the proof of Proposition~\ref{casiultima}, $\mathbb{S}(k_{\delta_1},\ldots,k_{\delta_\ell})\cong \overline{e_L}$ is a retract of $X$ (c.f.~\cite[proof of Proposition~4]{FelixTanre09}).}
}\end{remark}

\colorao{It well known that, under suitable normality conditions, the higher topological complexity of a Cartesian product can be estimated by
\begin{equation}\label{laestimacion}
\zcl_s(H^*(X))+\zcl_s(H^*(Y))\leq \zcl_s(H^*(X\times Y))\leq \TC_s(X \times Y) \leq \TC_s(X) + \TC_s(Y),
\end{equation}
see~\cite[Proposition 3.11]{bgrt} and~\cite[Lemma~2.1]{CoFa11}. Of course, these inequalities are sharp provided $\TC_s=\zcl_s$ for both $X$ and $Y$. In particular, for the spaces dealt with in \colorito{Theorem~\ref{TC_s(X)},} $\TC_s$ is additive in the sense that the higher topological complexity of a Cartesian product is the sum of the higher topological complexities of the factors. This generalizes the known $\TC_s$-behavior of products of spheres, see~\cite[Corollary~3.12]{bgrt}. However, if Cartesian products are replaced by wedge sums, the situation becomes much subtler. To begin with, we remark that Theorem~3.6 and Remark~3.7 in~\cite{MR3267004}, together with~\cite[Theorem~19.1]{farber06}, give evidence suggesting that a reasonable wedge-substitute of~(\ref{laestimacion}) (for $s=2$) would be given by
$$
\max\{\TC_2(X),\TC_2(Y),\cat(X\times Y)\}\leq\TC_2(X\vee Y)\leq\max\{\TC_2(X),\TC_2(Y),\cat(X)+\cat(Y)\}.
$$
We show that both of these inequalities hold as equalities for the spaces dealt with in the previous section (c.f.~\cite[Proposition~3.10]{CohenPruid08}). More generally:}

\begin{proposition}\label{farberdranishnikov}
\colorao{Let $X$ and $Y$ be subcomplexes of $\hspace{.6mm}\mathbb{S}(k_1 \ldots, k_n)$ and $\hspace{.6mm}\mathbb{S}(k_{n+1}, \ldots, k_{n+m})$ respectively. If $\cat (X) \geq \cat (Y)$ and all the $k_i$ have the same parity, then} 
$$\colorao{\TC_s(X \vee Y)= \max \{\TC_s(X), \TC_s(Y), \cat(X^{s-1})+ \cat (Y)\}.}$$
\end{proposition}
\begin{proof}
\colorao{If all the $k_i$ are even, the conclusion follows directly from Theorem~\ref{TC_s(X) even} and Remark~\ref{downtocat} at the end of the paper. In fact $\TC_s(X\vee Y)=\TC_s(X)$ under the present hypothesis.}

\smallskip
\colorao{Assume now that all the $k_i$ are odd, and think of $X \vee Y$ as a subcomplex of  $X \times Y$ inside $\mathbb{S}(k_1, \ldots, k_n, k_{n+1}, \ldots, k_{n+m})$, so that $\mathcal{K}_{X\vee Y}$ \colorito{is the disjoint union of $\mathcal{K}_{X}$ and $\mathcal{K}_{Y}$.} Since $\cat (X)=\dim(\mathcal{K}_X)+1\geq\cat (Y)=\dim( \mathcal{K}_Y)+1$, for maximal simplexes $J_1,\ldots,J_s$ of $\mathcal{K}_{X\vee Y}$ we see  
\begin{equation}\label{word4word}
\colorito{\N}_{X \vee Y} (J_1, \ldots, J_s) \leq \begin{cases}
\TC_s(X), & \text{if} \,\, J_1, \ldots, J_s \subseteq [n] ; \\
\TC_s(Y), & \text{if} \,\, J_1, \ldots, J_s \subseteq \{n+1, \ldots, n+m\} ; \\
(s-1)\cat (X)+ \cat (Y), & \text{otherwise}.
\end{cases}
\end{equation}
Therefore $\TC_s(X \vee Y) \leq \max \{\TC_s(X), \TC_s(Y), \, (s-1)\cat(X)+ \cat (Y)\}$. The reverse inequality holds since each of $\TC_s(X)$, $\TC_s(Y)$, and $(s-1) \cat (X)+ \cat (Y)$ can be achieved as a $\N_{X \vee Y} (J_1, \ldots, J_s)$ for a suitable combination of maximal simplexes $J_i$ of $\mathcal{K}_{X\vee Y}$.}
\end{proof}

\section{\colorito{The unrestricted} case}
\colorito{We now prove Theorem~\ref{TC_s(X)} in} the general case, \colorito{that is for} $X$ a subcomplex of $\hspace{.3mm}\mathbb{S}(k_1, \ldots, k_n)$ where all \colorito{the} $k_i$ are positive integers with no restriction on their parity. \colorito{As usual, we start by establishing the upper bound.}

\subsection{Motion planner}\label{stn4.1}
\colorito{Consider the disjoint union decomposition $[n]= J_E \sqcup J_O $ where $J_E$ is the collection of indices $i\in[n]$ for which $k_i$ is even (thus $i \in J_O$ if and only if $k_i$ is odd)}. For a subset $K \subseteq J_E$ and $P \in \mathcal{P}$, \colorito{let $X^s_{P, K}\subseteq X^s$ and $\N(P, K)$, the norm of $(P, K)$, be defined by}
\begin{itemize}
\item $X^s_{P, K}= X_P^s \cap \left\{(b_1, \ldots, b_s) \in X^s\; \Big|\hspace{-3.6mm} \begin{array}{ll} & \text{for \colorito{each $(i,j)\in K\times[s],\;\;$}} b_{ij}\neq \pm e^0, \,\,\,\, \colorito{\text{while}}\\ & \colorito{\text{for each } i \in J_E- K \text{ there is}\,\, j\in [s] \text{ with } b_{ij}= \pm e^0}\!\!
 \end{array}  \right\}$
\item $\N(P, K)= |P|+ |K|$ where $|P|$ is defined in~\colorito{(\ref{lanormadeP}).} 
\end{itemize}
\colorito{This extends the definitions of $X^s_{P, k}$ and $\N(P, k)$ done when all the $k_i$ are even.}

\smallskip As in the cases where all \colorito{the $k_i$} have the same parity, the \colorito{higher topological complexity of a} subcomplex $X$ of $\mathbb{S}(k_1, \ldots, k_n)$, \colorito{now} with no restrictions on the parity of the sphere \colorito{factors}, is encoded just by \colorito{the} combinatorial information \colorito{on the cells} of $X$. Consider
\begin{equation}\label{sjdhfuxji}
\mathcal{N}^s(X)= \max \left\{\left.\N_X(J_1, \ldots, J_s) + \Big|\bigcap_{i=1}^s J_i \cap J_E \Big|\;\; \right| \,\, \colorito{J_1,\ldots,J_s\in\mathcal{K}_X} \right\}
\end{equation}
where $\N_X(J_1, \ldots, J_s)$ is defined in~(\ref{nx}) for $\mathcal{K}=\mathcal{K}_X$. \colorito{Since both $\N_X(J_1, \ldots, J_s)$ and $|\bigcap_{i=1}^s J_i \cap J_E|$ are monotonically non-decreasing functions of the $J_i$'s, the definition of $\mathcal{N}^s(X)$ can equally be given using only maximal simplexes $J_i\in\mathcal{K}_X$.} \colorito{Further, by~(\ref{sy}),} $\mathcal{N}^s(X)$ can be rewritten as
\begin{equation}\label{gtc_sX}
\mathcal{N}^s(X)= \max \left\{\displaystyle{\sum_{i=1}^s |J_i| - \Big| \bigcap_{i=1}^s J_i \cap J_O \Big|} \,\, | \,\, e_{J_i} \,\, \text{is a cell of } \,\, X, \, \text{for all}\,\, i \in [s] \right\}.
\end{equation}

\begin{theo}\label{TC_s(X), gene}
\colorito{For} a subcomplex $X$ of $\hspace{.4mm}\mathbb{S}(k_1, \ldots, k_n)$, $$\TC_s(X)= \mathcal{N}^s(X).$$
\end{theo}

\colorito{Theorem~\ref{TC_s(X), gene} generalizes Theorems~\ref{tcsX odd} and~\ref{TC_s(X) even}. This is obvious} when all \colorito{the $k_i$} are odd \colorito{for then both $\mathcal{N}^s(X)$ and $\N^s(X)$ agree with} $$\max \left\{\left.\displaystyle{\sum_{i=1}^s |J_i| - \Big| \bigcap_{i=1}^s J_i  \Big|} \,\, \right| \,\, \colorito{J_1,\ldots,J_s\in\mathcal{K}_X} \right\},$$ \colorito{whereas if all the $k_i$} are even, $$\mathcal{N}^s(X)= \max \left\{\displaystyle{\sum_{i=1}^s |J_i| } \,\, \big| \,\, \colorito{J_1,\ldots,J_s\in\mathcal{K}_X} \right\} =s (1+ \dim \mathcal{K}_X).$$

\colorito{The estimate $\mathcal{N}^s(X)\leq\TC_s(X)$ in Theorem~\ref{TC_s(X), gene} will be proved in the next subsection by extending the cohomological methods in Section~\ref{genzero}. Here we prove the estimate $\TC_s(X)\leq\mathcal{N}^s(X)$ by constructing} an optimal motion planner \colorito{with $\mathcal{N}^s(X)+1$ local rules. The corresponding local domains will be obtained} by clustering subsets $X^s_{P, K}$ \colorito{for which the pair} $(P, K)\in \mathcal{P}\times 2^{J_E}$ \colorito{has a fixed} norm. \colorito{In detail,} for $j \in [\mathcal{N}^s(X)]_0$ \colorito{let}
\begin{equation}\label{lagejotatopologicamente}
G_j:= \bigcup_{\N(P, K)=j} X^s_{P, K}.
\end{equation}

\begin{lema}\label{lemma4.2pairwisedisjointcovering}
\colorito{The sets $G_0,\ldots,G_{\mathcal{N}^s(X)}$ yield a pairwise disjoint covering of $X^s$.} 
\end{lema}
\begin{proof}
\colorito{It is easy to see that $G_{j}\cap G_{j'}=\emptyset$ for $j\neq j'$.} Let $b=(b_1, \ldots, b_s) \in e_{J_1} \times \cdots \times e_{J_s}  \subseteq X^s$, where $J_j \subseteq [n]$ for $j \in [s]$. As in Lemma \ref{coverodd}, we have
\begin{equation}\label{cg1nueva}
\sum_{i=1}^n |\{G\cdot b_{ij}\,\, |\,\, j\in[s]\}|-n \leq  \sum_{j=1}^s \, |J_{j}| - \Big|\bigcap_{j=1}^s J_j \Big|= \N_X(J_1, \ldots, J_s).
\end{equation}
Moreover, it is clear that
\begin{equation}\label{cg2nueva}
\Big| \left\{ \colorito{i \in J_E} \; | \,\, b_{ij} \neq \pm e^0, \, \, \forall j\in [s]\right\} \Big| \leq  \Big| \bigcap_{i=1}^s J_i \cap J_E \Big|.
\end{equation}
Thus, if $P \in \mathcal{P}$ is the type of $b$, and \colorito{$K \subseteq J_E$ is determined by the condition that} $b \in X^s_{P, K}$, \colorito{then $N(P,K)=|P|+|K|\leq\mathcal{N}^s(X)$ in view of~(\ref{Pcardinality}), (\ref{cg1nueva}) and~(\ref{cg2nueva}).}
\end{proof}

\begin{lema}\label{emptyinterg}
\colorito{$(\ref{lagejotatopologicamente})$ is a topological disjoint union. Indeed,}
\begin{equation}\label{disuniong}
X^s_{P,K} \cap \overline{X^s_{P', K'}}= \emptyset=\overline{ X^s_{P,K}} \cap X^s_{P', K'}
\end{equation}
 for $(P, K), (P', K') \in \mathcal{P} \times 2^{J_E}$ provided that $(P, K) \neq (P', K')$ and $\N(P, K)=\N(P', K')$.
\end{lema}
\colorito{The} following \colorito{observation} will be useful \colorito{in the proof of Lemma~\ref{emptyinterg}:}

\begin{remark}\label{noindex}{\em
Let $K, K' \subseteq 2^{J_E}$ and \colorito{$P, P' \in \mathcal{P}$. If} there exists an index $i \in \colorito{K-K',}$ \colorito{then}
\begin{itemize}
\item $b_{ij} \neq  \pm e^0$ for all $j \in [s]$ \colorito{provided} $b=(b_1, \ldots, b_s)\in X^s_{P, K}$.
\item $b_{ij_0}= \pm e^0$ for some $j_0 \in [s]$ \colorito{provided} $b=(b_1, \ldots, b_s) \in X^s_{P', K'}$.
\end{itemize}
Therefore, $X^s_{P,K} \cap \overline{X^s_{P', K'}}= \emptyset$.
}\end{remark}

\begin{proof}[Proof of Lemma~$\ref{emptyinterg}$] \colorito{There are three possibilities:}

\medskip
\noindent \textbf{Case} $K=K'$. In this case, one conclude that $P \neq P'$ with $|P|=|P'|$, since $(P, K) \neq (P', K')$ and $\N(P, K)=\N(P', K')$. The desired equalities follow from Proposition \ref{interxp}.

\medskip
\noindent\textbf{Case} $P=P'$. In this case we have $K \neq K'$ with $|K|=|K'|$. Then, there \colorito{exist} indexes $i, i'\colorito{\in[n]}$ such that $i \in \colorito{K - K}'$ and $i' \in \colorito{K'-K.}$ Therefore, equalities~(\ref{disuniong}) follow from Remark \ref{noindex}.

\medskip
\noindent\textbf{Case} $P\neq P'$ and $K\neq K'$. \colorito{Without loss of generality we can assume} $|P|> |P'|$. Then there exists $i \in [n]$ such that $|P_i|> |P_i'|$, thus $X^s_{P,K} \cap \overline{X^s_{P', K'}}= \emptyset$. Moreover, \colorito{since} $|K| < |K'|$ \colorito{is forced,} there exits $i \in \colorito{K' - K,}$ \colorito{so that} $\overline{X^s_{P,K}} \cap X^s_{P', K'}= \emptyset$ by \colorito{Remark~\ref{noindex}.}
\end{proof}

\colorito{Lemmas~\ref{lemma4.2pairwisedisjointcovering} and~\ref{emptyinterg} reduce the proof of Theorem~\ref{TC_s(X), gene} to checking that each $X^s_{P,K}$ is an ENR admitting a local rule. Thus, troughout the remaining of this subsection we fix a pair $(P, K) \in \mathcal{P} \times 2^{J_E}$ with $P=(P_1, \ldots, P_n)$ and where each $P_i=\{\alpha_1^i, \ldots, \alpha^i_{n(P_i)}\}$ is assumed to be ordered as indicated at the beginning of Section~\ref{secciondosmeramera}.}

\medskip \colorito{Our analysis of $X^s_{P,K}$ depends on establishing a topological decomposition of $X^s_{P,K}$. To start with, note the topological} disjoint union decomposition $$X^s_{P,K}=\,\bigsqcup_{\beta} \,X^s_{P, K} \cap X^{\colorito{s}}_{P, \beta}$$ where the union runs over all $\beta=(\beta^1, \ldots, \beta^n)$ as in~(\ref{X^s_P}). \colorito{But we need a further splitting of each term $X^s_{P, K} \cap X^s_{P, \beta}$.}

\medskip Let $I=\{\ell_1, \ldots, \ell_{|I|}\}$ \colorito{stand for} $J_E- K$ and, for each $i \in [n]$, consider the subsets \colorito{$T_0(\alpha_1^i)$ and $T_1(\alpha_1^i)$ defined in~(\ref{T_01}).} \colorito{For each} $\epsilon=(t_1, \ldots, t_{|I|}) \in \{0, 1\}^{|I|}$ \colorito{define} $$\colorito{T_{\epsilon}}=\bigcap_{i = 1}^{|I|} T_{t_i}( \alpha_1^{\ell_i}).$$ \colorito{We then get a topological disjoint union decomposition} $$X^s_{P, K} \cap X^s_{P, \beta}\,= \,\bigsqcup_{\colorito{\epsilon \in \{0, 1\}^{|I|}}} \, X^s_{P, K} \cap X^s_{P, \beta} \cap\colorito{T_{\epsilon}.}$$ Therefore, \colorito{the updated task is the proof of:} 

\begin{lema}\label{updatedtarea}
\colorito{Each $ X^s_{P, K, \beta, \epsilon} := X^s_{P, K} \cap X^s_{P, \beta} \cap \colorito{T_{\epsilon}}$ is an ENR admitting a local rule.}
\end{lema}
\begin{proof} 
\colorito{The ENR assertion follows just as in the first paragraph of the proof of Proposition~\ref{construfinal}. The construction of the local rule is also similar to the those at the end of Subsections~\ref{suboddmotion} and~\ref{laspares}, and we provide the generalized details for completeness.}

\smallskip For $i=0,1$ and $j=0, 1, 2$, let $\tau_i$ and $\rho_j$ be the local rules, \colorito{with corresponding local domains $A_i$ and $B_j$,} obtained \colorito{in Subsections~\ref{suboddmotion} and~\ref{laspares}} by normalizing the local rules $\phi_i$ and $\lambda_j$ given in Examples \ref{TCs S} and \ref{TCs S even} \colorito{---see the proof of Proposition~\ref{construfinal} and the considerations following the proof of Proposition~\ref{interxp}.}

\medskip
As \colorito{before,} it is useful to keep in mind that elements $(\rojo{b}_1, \ldots, \rojo{b}_s) \in X^s$, with $\rojo{b}_j=(\rojo{b}_{1j}, \ldots, \rojo{b}_{nj})$ for $j\in \rojo{[s]}$, \rojo{can be thought of as matrices $(b_{i,j})$ whose columns represent the various stages in $X$ through which motion is to be planned (necessarily along rows). Again, we follow a pivotal strategy. In detail, in terms of the notation set at the beginning of the introduction for elements in the function space $X^{J_s}$,} consider the map
\begin{equation}\label{elplanmotbis}
\rojo{\green{\varphi}\colon X^s} \to \mathbb{S}\rojo{(k_1,\ldots, k_n)}^{J_s}
\end{equation}
\rojo{given by} $\green{\varphi}\left((\rojo{b}_1, \ldots, \rojo{b}_s)\right)=\left(\green{\varphi_1}(\rojo{b}_1, \rojo{b}_1), \ldots , \green{\varphi_s}(\rojo{b}_1, \rojo{b}_s)\right)$ where, \rojo{for $j\in[s]$,} $$\green{\varphi_j}(\rojo{b}_1, \rojo{b}_j)= \left(\green{\varphi_{1j}} (\rojo{b}_{11}, \rojo{b}_{1j}), \ldots, \green{\varphi_{nj}}(\rojo{b}_{n1}, \rojo{b}_{nj})\right)$$ \rojo{is the path in $\mathbb{S}(k_1,\ldots,k_n)$, from $b_1$ to $b_j$, whose $i$-th coordinate} $\green{\varphi_{ij}}(\rojo{b}_{i1}, \rojo{b}_{ij})$, \rojo{\rojo{$i\in[n]$,} is the path in $S^{k_i}$, from $b_{i1}$ to $b_{ij}$,} defined by $$\green{\varphi_{i,j}} (\rojo{b}_{i1},\rojo{b}_{ij})(t)=\begin{cases}   \rojo{b}_{i1}, & 0 \leq t \leq t_{ \rojo{b}_{i1}}, \\ \rojo{\sigma}(\rojo{b}_{i1}, \rojo{b}_{ij})(t-t_{\rojo{b}_{i1}}), & t_{\rojo{b}_{i1}} \leq t \leq 1.  \end{cases}$$ Here $t_{\rojo{b}_{i1}}=\frac{1}{2} -d(\rojo{b}_{i1}, \rojo{e^0})$ and 
\begin{equation}\label{lasigmabis}
\rojo{\sigma(b_{i1}, b_{ij})}=\begin{cases}\tau_0\rojo{(b_{i1}, b_{ij})}, & \colorito{\text{if } i\in J_O \text{ and }} (\rojo{b}_{i1}, \rojo{b}_{ij}) \in A_0; \\ \tau_1\rojo{(b_{i1}, b_{ij})}, & \colorito{\text{if } i\in J_O \text{ and }} (\rojo{b}_{i1}, \rojo{b}_{ij}) \in A_1;\\
\rho_0(b_{i1}, b_{ij}), & \colorito{\text{if } i\in J_E \text{ and }} (b_{i1}, b_{ij}) \in B_0;\\
\rho_1(b_{i1}, b_{ij}), & \colorito{\text{if } i\in J_E \text{ and }}(b_{i1}, b_{ij}) \in B_1;\\ 
\rho_2(b_{i1}, b_{ij}), & \colorito{\text{if } i\in J_E \text{ and }}(b_{i1}, b_{ij}) \in B_2. \end{cases}
\end{equation}

\rojo{Although \green{$\varphi$} is not continuous, its restriction $\green{\varphi}_{P, K, \beta,  \epsilon}$ to $X^s_{P, K, \beta, \epsilon}$ is, for then~(\ref{lasigmabis}) takes the form}
$$
\sigma=\begin{cases}\tau_1, & \mbox{$\colorito{i\in J_O,}\,\,\, j \notin \alpha^i_{1}$ or $j \in \rojo{\beta^i\cup\{1\}}$}; \\ \tau_0, & \mbox{$\colorito{i\in J_O,}\,\,\, j \in \alpha^i_{1}$ and $j \notin \rojo{\beta^i\cup\{1\}}$};\\ 
\rho_2, & i\in J_E,\,\,\roji{j \notin \alpha^i_{1}\,\, \text{or} \,\, j \in \rojo{\beta^i\cup\{1\}}};\\
\rho_1, & i\in J_E,\,\, j\in \alpha_1^i - \left(\beta^i\cup\{1\}\right)\text{ and } t_i=0;\\
\rho_0, & i\in J_E,\,\, j \in \alpha_1^i - \left(\beta^i\cup\{1\}\right)\text{ and } t_i=1. \end{cases}
$$
\colorito{Moreover,} $\green{\varphi}_{P, K, \beta,  \epsilon}$ is clearly a section for $e_s^{\rojo{\mathbb{S}(k_1,\ldots,k_n)}}$, \colorito{while the} fact that $\varphi_{P, K, \beta,  \epsilon}$ actually takes values in $X^{J_s}$ is \colorito{verified with an argument identical to the one proving} Proposition~\ref{noleave}.
\end{proof}

\subsection{Zero-divisors cup-length}\label{genzero}
We \colorito{next} show that, for a subcomplex $X$ of $\mathbb{S}(k_1,\ldots,k_n)$ (with no restrictions on the parity of the $k_i$, $i \in [n]$), the cohomological lower bound for $\TC_s(X)$ in Proposition~\ref{ulbTCn} is optimal and agrees with the upper bound coming from our explicit motion \colorito{planner} in the previous \colorito{subsection. Here} we use same considerations and notation as in Section \ref{sectres}.

\begin{proposition}\label{cotaabajoimparbis}
\colorito{A subcomplex $X$ of $\,\mathbb{S}(k_1,\ldots,k_n)$ has} $$\mathcal{N}^{\rojo{s}}(X)\leq\zcl_s(H^*(X)).$$
\end{proposition}
\begin{proof}
\colorito{We use the tensor product ring $H_X$, and the elements $u(\ell)\in H_X$ for} $u \in H^*(X)$, \colorito{as well as the elements $\gamma(J_1,\ldots,J_\ell)\in H_X$ for $J_1,\ldots,J_\ell\in\mathcal{K}_X$ defined for $2 \leq \rojo{\ell} \leq s$ at the beginning of the proof of Proposition~\ref{cotaabajoimpar} (but this time we will only need the latter elements in the range $3\leq\ell\leq s$). In addition,} let $J'= \bigcap_{j=1}^s J_j \cap J_E$ and consider 
\begin{eqnarray}
\bar{\epsilon}_{J'}&=& \prod_{j \in J'} (\epsilon_j \otimes 1 \otimes \cdots \otimes 1-1 \otimes \epsilon_j \otimes 1 \otimes \cdots \otimes 1)^2\label{elemento34decorrecion}\\ 
&=& (-2)^{|J'|} \epsilon_{J'} \otimes \epsilon_{J'} \otimes 1 \otimes \cdots \otimes 1 \nonumber
\end{eqnarray}
and
\begin{eqnarray}
\bar{\gamma}(J_1, J_2)&=& \displaystyle{\prod_{j \in (J_1-J_2) \cup (J_2-J')} \epsilon_j (2)} 
\label{elemento35decorreccion} \\ \nonumber
&=& \displaystyle{ \sum_{\phi_{2} \subseteq  (J_1-J_2 )\cup (J_2-J')} \pm \epsilon_{\phi^c_{2}} \otimes \epsilon_{\phi_{2}} \otimes 1 \otimes \cdots \otimes 1}
\end{eqnarray}
\colorito{where, as in the proof of Proposition~\ref{cotaabajoimpar}, $\phi^c_2$ stands for the complement of $\phi_2$ in $ (J_1-J_2 )\cup (J_2-J')$. Then}
\begin{equation}\label{prodzclsbis}
\bar{\epsilon}_{J'}\cdot \bar{\gamma} (J_1, J_2) \cdot\prod_{\ell=3}^{s} \gamma (J_1, \ldots, J_{\ell})= \sum_{\phi_ 2, \ldots, \phi_s} \pm  2^{|J'|} \epsilon_{J'} \epsilon_{\phi_2^c} \cdots\epsilon_{\phi^c_s} \otimes \epsilon_{J'} \epsilon_{\phi_2} \otimes \colorito{\epsilon_{\phi_3}}\otimes\cdots \otimes \epsilon_{\phi_s}
\end{equation}
\colorito{where, for $3\leq\ell\leq s$,
$$
\phi_\ell \rojo{{}\subseteq{}} \Big(\bigcap_{m=1}^{\ell-1} J_m - J_{\ell} \Big) \cup J_{\ell}
$$
with $\phi^c_{\ell}$ standing for the complement of $\phi_{\ell}$ in $\Big(\bigcap_{m=1}^{\ell-1} J_m - J_{\ell} \Big) \cup J_{\ell}$ ---here we are using the notation in Proposition~\ref{cotaabajoimpar}. Recalling that} $$\N_X(J_1, \ldots, J_s)=\sum_{\ell=2}^s \left(\Big| \bigcap_{m=1}^{\ell-1} J_m- J_{\ell} \Big|+ \Big| J_{\ell} \Big| \right),$$ \colorito{we easily see that the left-hand side of~(\ref{prodzclsbis})} is a product of $\N_X(J_1, \ldots, J_s)+ |\bigcap_{j=1}^sJ_j \cap J_E|$ zero-divisors. \colorito{Thus, by~(\ref{sjdhfuxji}),} it suffices to prove the non-triviality of the \colorito{right-hand side of~(\ref{prodzclsbis}).} \rojo{With this in mind, note that} the term
\begin{equation}\label{elterminobis}
\pm 2^{|J'|} \, \epsilon_{J'} \, \epsilon_{J_1- J_2} \, \colorito{\epsilon_{(J_1 \cap J_2)-J_3}} \cdots \epsilon_{(J_1 \cap \cdots \cap J_{s-1})-J_s} \otimes \epsilon_{J_2} \otimes \cdots \otimes\epsilon_{J_s},
\end{equation}
\rojo{which} appears in~(\ref{prodzclsbis}) \rojo{with} \,$\phi_\ell=J_\ell$ for $3 \leq  \ell \leq s$ and $\phi_2 =J_2-J'$, is \rojo{a basis element} because $$ \epsilon_{J'} \cdot\epsilon_{J_1- J_2}  \cdots \epsilon_{(J_1 \cap \cdots \cap J_{\ell-1})-J_\ell} \cdots \epsilon_{(J_1 \cap \cdots \cap J_{s-1})-J_s}= \epsilon_{J'}\cdot \epsilon_{(J_1- \cap_{j=1}^s J_j)}= \epsilon_{J_0}$$ with $J_0 \az{{}\subseteq{}} J_1$. \rojo{The non-triviality of~(\ref{prodzclsbis}) then follows by observing that~(\ref{elterminobis})} cannot arise when other summands \rojo{in~(\ref{prodzclsbis})} are expressed in terms of \rojo{the} basis for $H_X$. \rojo{In fact, each summand
\begin{equation}\label{otrossumandosbis}
\pm 2^{|J'|} \epsilon_{J'} \epsilon_{\phi_2^c}  \cdots \epsilon_{\phi^c_s} \otimes \epsilon_{J'}  \epsilon_{\phi_2} \otimes \colorito{\epsilon_{\phi_3}} \otimes\cdots \otimes \epsilon_{\phi_s}
\end{equation}
in~(\ref{prodzclsbis}) is either zero or a basis element and, in the latter case,~(\ref{otrossumandosbis}) agrees (up to sign) with~(\ref{elterminobis}) only if $\phi_{\ell}=J_{\ell}$ for $\ell=3,\ldots,s$, and $\phi_2= J_2-J'$.}
\end{proof}

\begin{remark}{\em
\colorito{The factors~(\ref{elemento34decorrecion}) and~(\ref{elemento35decorreccion}) adjust the product~(\ref{prodzcls}) of zero divisors in the proof of Proposition~\ref{cotaabajoimpar} so to account for the differences noted in Remark~\ref{elhandlingparodd}.}
}\end{remark}

\colorito{We close the section by noticing that Proposition~\ref{farberdranishnikov} holds without restriction on the parity of the sphere dimensions $k_1,\ldots,k_{n+m}$. \roji{That is:
\begin{proposition} Let $X$ and $Y$ be subcomplexes of $\hspace{.6mm}\mathbb{S}(k_1 \ldots, k_n)$ and $\hspace{.6mm}\mathbb{S}(k_{n+1}, \ldots, k_{n+m})$ respectively. If $\cat (X) \geq \cat (Y)$, then
$$\colorao{\TC_s(X \vee Y)= \max \{\TC_s(X), \TC_s(Y), \cat(X^{s-1})+ \cat (Y)\}.}$$ \end{proposition}}
The argument given in the second paragraph of the proof of Proposition~\ref{farberdranishnikov} applies word for word in the unrestricted case (replacing, of course, $\N_{X \vee Y} (J_1, \ldots, J_s)$ by $\sum_{i=1}^s|J_i|-|\bigcap_{i=1}^sJ_i\cap J_O|$ in~(\ref{word4word}) and in the last line of that proof).}

\section{Other polyhedral product spaces}\label{othersec}
\rojo{Polyhedral product spaces have recently been the focus of intensive research in connection to toric topology and its applications to other fields. In this section we determine the higher topological complexity of polyhedral product spaces $Z(\{(X_i,\star)\},\mathcal{K})$ for which each factor space $X_i$ admits a {\it $\TC_s$-efficient homotopy cell decomposition,} concept that is defined next.}

\medskip
\rojo{Recall that the spherical cone length of a path connected space $Y$, denoted here by $\cone(Y)$, is the least nonnegative integer $c$ for which there is a \emph{length-$c$ homotopy cell decomposition} $(Y_0,\ldots,Y_c)$ of $Y$, that is, a nested sequence of spaces $Y_0\subseteq\cdots\subseteq Y_c$ so that $Y_0$ is a point (the base point of all the $Y_i$'s), $Y_c$ has the (based) homotopy type of $Y$ and, for $0\leq i<c$, $Y_{i+1}$ is the (reduced) cone of a (based) map $\pi_i\colon W_i\to Y_i$ whose domain $W_i$ is a finite wedge of spheres (of possibly different dimensions). In such a situation, we refer to $Y_i$, to $Y_i-Y_{i-1}$, and to $\pi_i$, respectively, as the $i$-th layer, the $i$-th stratum, and the $i$-th attaching map of the homotopy cell decomposition. If no such integer $c$ exists, we set $\cone(Y)=\infty$. In these terms we say that $Y$ admits a $\TC_s$-efficient homotopy cell decomposition when $\TC_s(Y)=s\cone(Y)$. The adjective ``$\TC_s$-efficient'' is motivated by the following standard fact:}

\begin{lema}\label{decomposicionminimal}
\rojo{For a path connected space $X$, $\TC_s(X)\leq s\cone(X)$.}
\end{lema}

\rojo{The proof of Lemma~\ref{decomposicionminimal} given below makes use of products of homotopy cell decompositions, which is a standard construction in view of the finiteness condition on the number of cells in a given strata. For instance, the product of two homotopy cell decompositions $(Y_0,\ldots,Y_c)$ and $(Z_0,\ldots,Z_d)$, of $Y$ and $Z$ respectively, is the homotopy cell decomposition of $Y\times Z$ given by the sequence $(P_0,\ldots,P_{c+d})$ with $P_i=\bigcup_{j+k=i} Y_j\times Z_k$ and where we take the usual (Cartesian product) attaching maps.}

\begin{proof}[Proof of Lemma~$\ref{decomposicionminimal}$]
\rojo{Let $(X_0,\ldots,X_c)$ be a minimal homotopy cell decomposition of $X$. The product decomposition on $X^s$ has length $sc$, so the results follows from the fact that the sectional category of a fibration is bounded from above by the spherical cone length of its base.}
\end{proof}

\rojo{Known examples of spaces admitting a $\TC_s$-efficient homotopy cell decomposition are:}
\begin{enumerate}
\item \rojo{Wedge sums of spheres (with the single exception of a wedge with a single summand given by an odd dimensional sphere).}
\item \rojo{Simply connected closed symplectic manifolds admitting a cell structure with no odd dimensional cells.}
\item \rojo{Configuration spaces on odd dimensional Euclidean spaces.}
\end{enumerate}
\rojo{All such examples satisfy, in addition, the equality $\TC_s=\zcl_s$, a condition that will be part of Theorem~\ref{pps} below. In particular, our result implies that the list of examples above can be extended to polyhedral product spaces constructed from the three types of spaces already listed.}

\begin{definition}
\rojo{For an $n$ tuple $\azz{\gamma=(c_1,\ldots,c_n)}$ of nonnegative integers, we define the \azz{\emph{$\gamma$-weighted} dimension} of an abstract simplicial complex $\mathcal{K}$ with vertices $[n]$ as}
$$
\rojo{\azz{\dim_{\gamma}}(\mathcal{K})=\max\left\{ c_{i_1}+\cdots+c_{i_\ell} \,\, | \,\, 1\leq i_1<\cdots<i_\ell\leq n\mbox{ \hspace{.5mm}and \hspace{.3mm}} \{i_1,\ldots,i_\ell\}\in\mathcal{K}\right\}-1.}
$$
\end{definition}

\rojo{Theorem~\ref{TC_s(X) even} is generalized by:}

\begin{theo}\label{pps}
\rojo{Let $X=Z(\{(X_i,\star)\},\mathcal{K})\subseteq\prod_{i=1}^nX_i$ be the polyhedral product space associated to a family of pointed spaces $X_1,\ldots,X_n$, and an abstract simplicial complex $\mathcal{K}$ with vertices $[n]$. Assume that, for \azz{each} $i\in[n]$,}
\begin{itemize}
\item \rojo{$\TC_s(X_i)=\zcl(H^*(X_i;\mathbb{Q}))$, and}
\item \rojo{$X_i$ admits a $\TC_s$-efficient $($and necessarily minimal, in view of Lemma~$\ref{decomposicionminimal})$ homotopy cell decomposition.}
\end{itemize}
\rojo{Then $X$ also satisfies the two hypothesis above and, in addition, $\,\TC_s(X)=s(1+\azz{\dim_{\gamma}}(\mathcal{K}))\,$ where $\azz{\gamma=(\cone(X_1),\ldots,\cone(X_n))}$.}
\end{theo}

\begin{proof}[Proof of Theorem~$\ref{pps}$]
\rojo{For $i\in[n]$ let $(X_{0,i},X_{1,i},\ldots,X_{c_i,i})$ be a $\TC_s$-efficient (and necessarily minimal, in view of Lemma~\ref{decomposicionminimal}) homotopy cell decomposition of $X_i$. By the homotopy invariance of the polyhedral product functor, we can assume that $X_i=X_{c_i,i}$ for all $i\in[n]$. Let $(P_0,\ldots,P_c)$ be the product homotopy cell decomposition on $\prod_iX_i$ where $c=\sum_ic_i$, and let $P'_i=P_i\cap X$ for $i\in[c]_0$. Note that $(P'_0,\ldots,P'_c)$ is a homotopy cell decomposition of $X$ for which $$P'_{1+\azz{\dim_{\gamma}}(\mathcal{K})}=P'_{2+\azz{\dim_{\gamma}}(\mathcal{K})}=\cdots=P'_c,$$ so $\TC_s(X)\leq s(1+\azz{\dim_{\gamma}}(\mathcal{K}))$ in view of Lemma~\ref{decomposicionminimal}. To see that this is an equality (so that $(P'_0\ldots,P'_{1+\azz{\dim_{\gamma}}(\mathcal{K})})$ is $\TC_s$-efficient), choose $1\leq i_1<\cdots<i_\ell\leq n$ with $\{i_1,\ldots,i_\ell\}\in\mathcal{K}$ and $c_{i_1}+\cdots+c_{i_\ell}=1+\azz{\dim_{\gamma}}(\mathcal{K})$, and note that
\begin{eqnarray*}
\TC_s(X) & \geq &\TC_s(X_{i_1}\times\cdots\times X_{i_\ell}) \\
&\geq& \zcl_s(H^*(X_{i_1}\times\cdots\times X_{i_\ell};\mathbb{Q})) \\
&\geq& \sum_{j=1}^\ell\zcl_s(H^*(X_{i_j};\mathbb{Q})) \\
&=& s\sum_{j=1}^\ell c_{i_j} \\
&=& s(1+\azz{\dim_{\gamma}}(\mathcal{K})).
\end{eqnarray*}
The second and third inequalities hold by Proposition~\ref{ulbTCn} and~\cite[Lemma~2.1]{CoFa11}, respectively, whereas the first inequality holds since, as explained in the first paragraph of the proof of Proposition~4 in~\cite{FelixTanre09}, $X_{i_1}\times\cdots\times X_{i_\ell}$ is (homeomorphic to) a retract of $X$. To complete the proof, note that, as above, $\zcl_s(H^*(X;\mathbb{Q}))$ is bounded from above by $\TC_s(X)$ and from below by $\zcl_s(H^*(X_{i_1}\times\cdots\times X_{i_\ell};\mathbb{Q}))$---and that the last two numbers agree.} 
\end{proof}

\begin{remark}\label{downtocat}{\em
\colorao{The methods of this section can be applied to describe the category of suitably efficient polyhedral products. For instance, without any restriction on the parity of the sphere dimensions $k_i$, any subcomplex $X$ of $\mathbb{S}(k_1,\ldots,k_n)$ has $\cat(X^s)=s(1+\dim(\mathcal{K}_X))$. This is just an example of a partial (but very useful) generalization of~\cite[Proposition~4]{FelixTanre09}.}
}\end{remark}


\bigskip\bigskip\bigskip\sc
Departamento de Matem\'aticas

Centro de Investigaci\'on y de Estudios Avanzados del IPN

Av.~IPN 2508, Zacatenco, M\'exico City 07000, M\'exico

{\tt jesus@math.cinvestav.mx}

{\tt 
bgutierrez@math.cinvestav.mx}

\bigskip
Department of Mathematics 

University of Oregon

Eugene, OR 97403, USA

{\tt yuz@uoregon.edu}

\end{document}